\documentclass[11pt]{amsart}

\usepackage{amsxtra,amssymb,amsthm,amsmath,amscd,url,xypic,mathrsfs,mathtools}
\usepackage[latin1]{inputenc}
\usepackage{fullpage}
\usepackage{supertabular}

\usepackage[
        colorlinks,
        backref,
]{hyperref}
\usepackage{url}
\numberwithin{equation}{section}

\usepackage{color}

\renewcommand{\leq}{\leqslant}
\renewcommand{\geq}{\geqslant}

\def\stacksum#1#2{{\stackrel{{\scriptstyle #1}}
{{\scriptstyle #2}}}}

\newcommand{\Cc}{\mathbf{C}}
\newcommand{\D}{\mathbf{D}}

\newcommand{\Zz}{\mathbf{Z}}

\newcommand{\calG}{\mathbf{G}}
\newcommand{\cO}{\mathcal{O}}
\newcommand{\cG}{\mathcal{G}}
\newcommand{\cT}{\mathcal{T}}
\newcommand{\Hh}{\mathbf{H}}
\newcommand{\Qq}{\mathbf{Q}}
\newcommand{\Fp}{\mathbf{F}}
\newcommand{\Tt}{\mathbf{T}}
\newcommand{\T}{\mathbf{T}}
\newcommand{\G}{\mathbf{G}}
\newcommand{\Gg}{\mathbf{G}}
\newcommand{\GG}{\mathbf{G}}

\newcommand{\Fbar}{\overline{\mathbf{F}}}

\newcommand{\Ee}{\mathbf{E}}

\newcommand{\proba}{\mathbf{P}}
\newcommand{\expect}{\mathbf{E}}

\newcommand{\charfun}{\mathbf{1}}

\newcommand{\ideal}[1]{\mathfrak{{#1}}}

\newcommand{\ra}{\rightarrow}
\newcommand{\lra}{\longrightarrow}

\newcommand{\fleche}[1]{\stackrel{#1}{\lra}}

\DeclareMathOperator{\Spec}{Spec}

\DeclareMathOperator{\Gal}{Gal}

\DeclareMathOperator{\Aut}{Aut}

\DeclareMathOperator{\Ad}{Ad}
\DeclareMathOperator{\Frob}{Frob}
\DeclareMathOperator{\GL}{GL}
\DeclareMathOperator{\SL}{SL}
\DeclareMathOperator{\Sp}{Sp}
\DeclareMathOperator{\SO}{SO}

\newcommand{\demi}{{\textstyle{\frac{1}{2}}}}

\DeclareMathSymbol{\gena}{\mathord}{letters}{"3C}
\DeclareMathSymbol{\genb}{\mathord}{letters}{"3E}

\def\kbar{\overline{k}}
\def\calT{\mathcal{T}}
\def\calP{\mathcal{P}}
\def\calG{\mathcal{G}}
\def\FF{\mathbf{F}}

\def\p{\mathfrak{p}}

\newenvironment{romanenum}{\hfill \begin{enumerate} }{\end{enumerate}}

\theoremstyle{plain}
\newtheorem{theorem}{Theorem}[section]

\newtheorem{lemma}[theorem]{Lemma}
\newtheorem{corollary}[theorem]{Corollary}

\newtheorem{proposition}[theorem]{Proposition}

\theoremstyle{remark}
\newtheorem{remark}[theorem]{Remark}

\theoremstyle{definition}

\newtheorem{example}[theorem]{Example}

\begin{document}

\title{Splitting fields of characteristic polynomials of random
  elements in arithmetic groups} 
  
  \author{F. Jouve}
\address{D\'epartement de Math\'ematiques\\
B\^atiment 425\\
Facult\'e des Sciences d'Orsay\\
Universit\'e Paris-Sud 11\\
F-91405 Orsay Cedex, France}
\email{florent.jouve@math.u-psud.fr}

\author{E. Kowalski}
\address{ETH Z\"urich - DMATH\\
  R\"amistrasse 101\\
  8092 Z\"urich\\
  Switzerland} \email{kowalski@math.ethz.ch}

\author{David Zywina}
\address{Department of Mathematics and Statistics, Queen's University,
Kingston, ON  K7L~3N6, Canada}
\email{zywina@mast.queensu.ca}
\urladdr{http://www.mast.queensu.ca/\~{}zywina}


\begin{abstract}
  We discuss rather systematically the principle, implicit in earlier
  works, that for a ``random'' element in an arithmetic subgroup of a
  (split, say) reductive algebraic group over a number field, the
  splitting field of the characteristic polynomial, computed using any
  faitfhful representation, has Galois group isomorphic to the Weyl
  group of the underlying algebraic group. Besides tools such as the
  large sieve, which we had already used, we introduce some
  probabilistic ideas (large deviation estimates for finite Markov
  chains) and the general case involves a more precise understanding
  of the way Frobenius conjugacy classes are computed for such
  splitting fields (which is related to a map between regular elements
  of a finite group of Lie type and conjugacy classes in the Weyl
  group which had been considered earlier by Carter and Fulman for
  other purposes; we show in particular that the values of this map
  are equidistributed).
\end{abstract}

\maketitle

\section{Introduction}
In earlier works, in particular~\cite{jkze8}, we have considered
particular cases of the following ``principle'': if $g$ is a
``random'' rational element in a connected split reductive group $\G$
over $\Qq$, embedded in some $\GL(m)$, then the splitting field of the
characteristic polynomial of $g$ should have Galois group isomorphic
to the Weyl group of $\G$.
\par
In this paper, we consider this question in much greater generality
than previously.  We are thus led to replace somewhat ad hoc arguments
with more intrinsic constructions, in particular in two areas: (1) in
characterizing the splitting field of the polynomials we construct,
which we relate to splitting fields of tori; (2) in the understanding
of the situation over finite fields, which is required for the sieve
argument we use to obtain strong bounds on the probability of having a
Galois group smaller than expected.  Moreover, to handle the reduction
to simply-connected groups, we need as input some ideas from Markov
chains (in particular, some large deviation estimates for finite
Markov chains).
\par
Let $k$ be a number field and denote by $\Zz_k$ its ring of integers.
Let $\G$ be a connected linear algebraic group defined over $k$.  We
may view it as a matrix group by fixing a faithful embedding
$\rho\colon \G \hookrightarrow \GL(m)$ defined over $k$.  For each
$g\in \G(k)$, let $k_g$ be the splitting field over $k$ of the
characteristic polynomial $\det(T-\rho(g)) \in k[T].$ The goal of this
paper is to describe the Galois group $\Gal(k_g/k)$ for a ``random''
$g$ in terms of the geometry of $\G$.
\par
We will only consider those $g$ belonging to a fixed arithmetic
subgroup $\Gamma$ of $\G$.  Recall that an \emph{arithmetic subgroup}
of $\G$ is a subgroup $\Gamma$ of $\G(k)$ for which $\rho(\Gamma)$ is
commensurable with $\rho(\G(k))\cap \GL(m,\Zz_k)$; this definition is
independent of $\rho$.  We shall assume that our arithmetic group
$\Gamma$ is Zariski dense in $\G$ (otherwise the structure of the
Galois groups $\Gal(k_g/k)$ should be governed by a smaller algebraic
group).
\par 
Our notion of ``random'' in this paper is to view $\Gamma$ as the
vertices of a Cayley graph and perform a long random walk on this
graph.  First choose a finite set $S$ that generates the group
$\Gamma$ (arithmetic groups are finitely generated, see~\cite[Th. 4.17
(2)]{platonov-rapinchuk}), such that $S$ is symmetric, i.e.,
$S=S^{-1}$.  We then have a \emph{Cayley graph} associated to the pair
($\Gamma$, $S$): the vertices are the elements of $\Gamma$ and there
is an edge connecting the vertices $g_1$ and $g_2 \in \Gamma$ if and
only if $g_1 g_2^{-1}$ belongs to $S$ (note we allow $1\in S$, in
which case the graph has self-loops at each vertex).  This graph is
regular of degree $|S|$.  Starting at the vertex $1\in \Gamma$ of our
graph, we take a random walk by repeatedly following one of the $|S|$
edges emanating from the current vertex with equal probability.  More
precisely, for each $n\geq 1$, we will choose a random element $s_n$
of $S$ (with uniform distribution); this gives a walk $X_0=1,\,
X_1=s_1,\, X_2=s_1s_2,\, X_3=s_1s_2s_3,\, X_4=s_1s_2s_3s_4, \ldots$.

\begin{theorem}\label{th-main-easy}
  Let $\G$ be a reductive group defined over a number field $k$, and
  fix a faithful representation $\rho\colon \G \hookrightarrow \GL(m)$
  defined over $k$.  Let $\Gamma\subseteq \G(k)$ be an arithmetic
  subgroup of $\G$ and assume that it is Zariski dense in $\G$.  Let
  $S$ be a finite symmetric set of generators for $\Gamma$ such that
  $1\in S$.  For any $w=(s_1,\ldots, s_n)\in S^n$, let $k_w/k$ be the
  splitting field of the characteristic polynomial
$$
\det(T-\rho(s_1\cdots s_n)) \in k[T]
$$
over $k$.  Then there is a finite group $\Pi(\G)$ which contains the
Weyl group $W(\G)$ as a normal subgroup such that the following hold: 
\begin{romanenum}
\item
The Galois group $\Gal(k_w/k)$ is always isomorphic to a
subquotient of $\Pi(\G)$.
\item
We have
$$
\lim_{n\to \infty} \frac{\big|\big\{ w=(s_1,\ldots, s_n)\in
  S^n\,:\,\Gal(k_w/k)\cong \Pi(\G)\big\}\big|}{|S^n|}= 1.
$$
\item If $\G$ is semisimple, then there exists a constant $c>1$ such
  that
$$
\frac{\big|\big\{ w=(s_1,\ldots, s_n)\in S^n\,:\,\Gal(k_w/k)\cong
  \Pi(\G)\big\}\big|}{|S^n|}= 1 + O(c^{-n})
$$
for all $n\geq 1$. 
\item Let $\kbar$ be an algebraic closure of $k$ and let $k_\G$ be the
  intersection of all the extensions $K\subseteq \kbar$ of $k$ for
  which $\G_K$ is split.  There exists a constant $c>1$ such that
$$
\frac{\big|\big\{ w=(s_1,\ldots, s_n)\in S^n\,: \, \Gal(k_\G k_w/k_\G)
  \cong W(\G)\big\}\big|}{|S^n|} = 1 + O(c^{-n})
$$
for all $n\geq 1$. 
\end{romanenum}
\par
The constants $c$ and the implied constants depend only on $\G$ and
the set $S$.
\end{theorem}

We shall explicitly describe the group $\Pi(\G)$ in \S\ref{sec-tori}.
If we assume that $\G$ is split, then $k_{\G}=k$ and $\Pi(\G)=W(\G)$.
See Theorem~\ref{th-main2} for a more general version where we allow
different distributions of the steps $s_n$ and a general connected
linear algebraic group $\G$ over $k$.

\begin{example}
  Here are some illustrations of our theorem.
\par
(1) Let $\G$ be either $\SL(n)$ or $\Sp(2g)$ where $n\geq 2$ and
$g\geq 1$.  We may identify $\G$ as a matrix group via the natural
representation into $\GL(n)$ or $\GL(2g)$, respectively. Let $k=\Qq$
and take for $\Gamma$ the arithmetic subgroup $\SL(n,\Zz)$ or
$\Sp(2g,\Zz)$ of $\G$, respectively. The Weyl groups are,
respectively, the symmetric group on $n$ letters and the group of
signed permutations on $g$ letters. In those cases (where $k_{\G}=\Qq$
and $\Pi(\G)=W(\G)$), Theorem~\ref{th-main-easy} was proved
in~\cite[Th. 7.12]{lsieve} when $k=\Qq$.
\par
(2) For an example with $\Pi(\G)\neq W(\G)$, let us take for $\G$ a
non-split form of the special orthogonal group $\SO(4)$ over
$\Qq$. Say, the group corresponds to the positive isometries of the
four-dimensional space endowed with the nondegenerate indefinite
quadratic form $Q(x_1,\ldots,x_4)=x_1^2+x_2^2-x_3^2-x_4^2$. Of course
$\G$ is split over $\Qq(i)$.
\par
The Weyl group of $\SO(4)$ is isomorphic to the Klein four
group. Indeed this group corresponds to the Weyl group of the root
system of type $D_2$. However, a ``generic'' $g\in \SO(4,\Zz)$ should
have a characteristic polynomial whose splitting field $k_g$ over
$\Qq$ has Galois group sitting in the exact sequence
 $$
 1\rightarrow W(\SO(4))\rightarrow \Gal(k_g/\Qq)\rightarrow
 \Gal(\Qq(i)/\Qq)\rightarrow 1\,.
  $$
  Therefore the ``right'' maximal Galois group is an extension of
  $\Zz/2\Zz$ by $W(\SO(4))$, and it is in fact the Weyl group of the
  root system of type $C_2$.
\par
(3) Parts (ii) and (iv) involve a subtlety that we overlooked in the
first version of this paper, and which was pointed out by
L. Rosenzweig: if $\G$ is reductive, and not semisimple, then in
general we can not claim that the convergence in (ii) occurs
exponentially fast (in contrast with (iii)). For instance, consider
$k=\Qq$, and take a hyperbolic element $g_0$ in $\SL(2,\Zz)$. Let $\G$
be the Zariski-closure of the infinite cyclic subgroup $g_0^{\Zz}$
generated by $g_0$, so that $\G$ is a non-split torus. Take also
$\Gamma=g_0^{\Zz}\subset \G(\Qq)$ and $S=\{g_0,1,g_0^{-1}\}$. Then for
$w=(s_1,\ldots,s_n)\in S^n$, $k_w$ can be either the quadratic field
generated by the eigenvalues of $g_0$, or $\Qq$ itself, the second
case happening exactly when $s_1\cdots s_n=1$. But if $s_i=g_0^{m_i}$
with $m_i\in\{-1,0,1\}$, the condition becomes $m_1+\cdots+m_n=0$, which
occurs with probability approximately $n^{-1/2}$ (by the Stirling
formula). 
\par
In the semisimple case our theorem provides exponential decay, in
terms of the ``length'' of the random walk, of the probability that
the Galois group is ``small''. In the general reductive case, one can
very likely also derive a general quantitative bound, though only with
polynomial decay, and it should be possible to characterize those
groups $\G$ for which one can recover exponential decay. 
\end{example}
   
\begin{remark}
(1)  There are some interesting connections between our results and ideas
  introduced by Prasad and Rapinchuk~\cite[\S 3]{prasad-rapinchuk} to
  study the relation of ``weak commensurability'' in arithmetic
  groups.
\par
(2) There are other ways to try to describe ``random'' elements in an
arithmetic group; we comment on these in Section~\ref{sec-comments},
and indicate in particular some interesting natural questions which
arise from the probabilistic construction we have chosen.
\end{remark}
 
The plan of the paper is as follows. In Section~\ref{sec-tori} we
analyze, in general, splitting fields of the type considered and
relate them with splitting fields of maximal tori in $\G$, which are
more intrinsic; this leads to a very general form of the a priori
inclusion which is part (i) of the theorem
above. Section~\ref{SS:reduction} is also of a preliminary nature and
discusses fairly standard facts on reduction of arithmetic groups
modulo primes. In Section~\ref{SS:finite fields}, we show that the
general construction, in this case, is closely related to earlier
results of Fulman~\cite{fulman} and Carter~\cite{carter-3}, and we
prove an equidistribution statement that will be useful for setting up
the sieve (and which is of independent interest). In
Section~\ref{sec-sieve}, we prove a general sieve result for
arithmetic subgroups of semisimple groups -- again, a result of
independent interest, where other deep ingredients come into play,
coming both from algebra (strong approximation results for arithmetic
groups) and from harmonic analysis (Property $(\tau)$).  Finally, in
Section~\ref{sec-final}, we combine the algebraic information with the
sieve result and some additional reduction steps in order to obtain
the general conclusion. In Section~\ref{sec-comments}, we compare our
approach with two other natural ways of quantifying the idea that
``random'' elements have the Weyl groups as Galois group.

\par
\medskip
\par
\textbf{Notation.} As usual, $|X|$ denotes the cardinality of a
set. 
For any integer $n\geq 1$, $\mathfrak{S}_n$ is the group of
permutations on $n$ letters.  For any group $G$, we denote by
$G^{\sharp}$ the set of conjugacy classes of $G$. We denote by $\Fp_q$
a field with $q$ elements. ``Connected'' will mean ``geometrically
connected'' for all algebraic groups considered.
By the \emph{type} of a connected reductive algebraic group $\G$
defined over a field $k$ (or a subring of $k$), we mean the
isomorphism type of its root datum over an algebraic closure of $k$
(see, e.g.,~\cite[\S 9.4]{springer}).
\par
By the Galois group of a polynomial, we mean the Galois group of its
splitting field. For a number field $k$, we denote by $\Zz_k$ the ring
of integers, and for $\ideal{p}$ a prime ideal of $\Zz_k$, we write
$\Fp_{\ideal{p}}$ for the residue field $\Zz_k/\ideal{p}\Zz_k$.
\par
For a scheme $X$ defined over a ring $A$ and a ring homomorphism $A\to
B$, we will denote the base extension $X\times_{\Spec A} \Spec B$ by
$X_B$. 
\par
\medskip
\par
\textbf{Acknowledgements.} Thanks to the referee for a thorough
reading, and thanks especially to L. Rosenzweig for both interesting
discussions related to this topic and for pointing out a serious
mistake in an earlier version.

\section{Splitting fields of tori and elements of algebraic groups}
\label{sec-tori}

In this section, we consider the Galois theory of splitting fields of
tori and elements in linear algebraic groups.  Throughout, let $\G$ be
a connected linear algebraic group defined over a perfect field $k$.

\subsection{Tori}

An algebraic group $\T$ over $k$ is a \emph{torus} if $\T_{\bar{k}}$
is isomorphic to $\GG^r_{\bar{k}}$ for an integer $r\geq 0$.

Fix a torus $\T$ defined over $k$.  We say that $\T$ is \emph{split}
if it is isomorphic over $k$ to $\GG_m^r$.  Denote by $X(\T)$ the
group of characters $\alpha\colon \T_{\bar{k}}\to \GG_{m,\bar{k}}$,
which is a free abelian group of rank equal to the dimension of $\T$.
There is a natural action of $\Gal(\kbar/k)$ on $X(\T)$ given by
$\sigma(\chi(t))={}^\sigma \chi(\sigma(t))$ for $\sigma\in
\Gal(\kbar/k)$, $\chi\in X(\T)$ and $t\in \T(\kbar)$.  Let
$k_{\T}\subseteq \kbar$ be the minimal extension of $k$ for which
$\Gal(\kbar/k_{\T})$ acts trivially on $X(\T)$; it is a finite Galois
extension of $k$ that we call \emph{the splitting field of $\T$}.  The
field $k_{\T}$ is also the minimal extension $K\subseteq \kbar$ of $k$
for which $\T_{K}$ is split.
  
Let $\varphi_\T \colon \Gal(\kbar/k)\to \Aut(X(\T))$ be the
representation describing the Galois action on $X(\T)$; we have
$\varphi_\T(\sigma) \chi = {}^\sigma\chi$ for all $\sigma\in
\Gal(\kbar/k)$ and $\chi \in X(\T)$.  It factors through an injective
homomorphism $\Gal(k_{\T}/k)\hookrightarrow \Aut(X(\T))$.

\subsection{Maximal tori}\label{galoistori}

Assume that $\G$ is reductive.  Let $\T$ be a maximal torus of $\G$,
defined over $k$ (we always consider maximal tori defined over the
base field).

In this section, we shall describe a finite subgroup $\Pi(\G)$ of
$\Aut(X(\T))$ that contains the image of $\varphi_\T$ and whose
isomorphism class depends only on $\G$.

Let $Z_\G(\T)$ and $N_\G(\T)$ be the
centralizer and normalizer, respectively, of $\T$ in $\G$.  The
\emph{Weyl group} of $\G$ with respect to $\T$, denoted $W(\G,\T)$, is
defined to be the $\kbar$-valued points of $N_\G(\T)/Z_\G(\T)$.  The
group $W(\G,\T)$ is finite. 

Conjugation induces an action of $W(\G,\T)$ on $\T$; for $w\in
W(\G,\T)$ represented by an element $n\in N_\G(\T)(\kbar)$, we have
$w\cdot t := n t n^{-1}$.  This action is faithful since $Z_\G(\T)=\T$
\cite[13.17 Corollary~2]{borel}.  The Weyl group $W(\G,\T)$ also acts
faithfully on $X(\T)$; for $\chi \in X(\T)$, $w\cdot \chi$ is the
character of $\T$ defined by $t\mapsto \chi(n^{-1} t n)$.  Using this
last action, we may identify $W(\G,\T)$ with a subgroup of
$\Aut(X(\T))$.

We define $\Pi(\G,\T)$ to be the subgroup of $\Aut(X(\T))$ generated
by $W(\G,\T)$ and $\varphi_\T(\Gal(\kbar/k))$.  Trivially, we have
$\varphi_\T(\Gal(\kbar/k)) \subseteq \Pi(\G,\T)$, so we may rewrite our
representation as
\[
\varphi_\T \colon \Gal(\kbar/k) \to \Pi(\G,\T).
\]
We will now show that the group $\Pi(\G,\T)$, up to isomorphism, is
independent of $\T$.

Let $\T_0$ be a fixed maximal torus of $\G$ defined over $k$.  Since
all maximal tori of $\G$ are conjugate over $\kbar$, there is an
element $x\in \G(\kbar)$ such that $\T_{\kbar} =
x\T_{0,\kbar}x^{-1}$. This gives isomorphisms $f\colon \T_{\kbar}
\xrightarrow{\sim} \T_{0,\,\kbar},$ $t \mapsto x^{-1} t x$ and
$F\colon X(\T) \xrightarrow{\sim} X(\T_0),$ $\chi \mapsto \chi \circ
f^{-1}$.

\begin{proposition}  \label{P:Weyl group}
With notation as above, the following hold:
\begin{romanenum}
\item
The Weyl group $W(\G,\T)$ is a normal subgroup of $\Pi(\G,\T)$.   
\item
 The map 
\begin{equation*} \label{E:Aut}
\Aut(X(\T))\xrightarrow{\sim} \Aut(X(\T_0)) 
\quad \gamma \mapsto F \circ \gamma \circ F^{-1}
\end{equation*} 
is an isomorphism which induces isomorphisms 
$$
\Pi(\G,\T)\xrightarrow{\sim} \Pi(\G,\T_0),\quad\quad
W(\G,\T)\xrightarrow{\sim} W(\G,\T_0).
$$
A different choice of $x$ gives the same isomorphisms up to
composition by an inner automorphism arising from an element of the
Weyl group.
\item Take $\sigma\in \Gal(\kbar/k)$ and let $w_{\sigma}$ be the
  element of $W(\G,\T)$ represented by $x^{-1}\sigma(x) \in
  N_\G(\T)(\kbar)$.  Then
\[
F\circ \varphi_\T(\sigma) \circ F^{-1} = w_{\sigma} \circ
\varphi_{\T_0}(\sigma).
\]
\item If $K\subseteq \kbar$ is an extension of $k$ for which $\G_{K}$
  is split, then $\varphi_\T(\Gal(\kbar/K))\subseteq W(\G,\T)$.
\end{romanenum}
\end{proposition}
\begin{proof}
  \noindent (i) For $\sigma \in \Gal(\kbar/k)$ and $w\in W(\G,\T)$, we
  need to show that $\varphi_\T(\sigma)\circ w \circ
  \varphi_\T(\sigma)^{-1}$ belongs to $W(\G,\T)$.  For a character
  $\chi \in X(\T)$, we have $$(\varphi_\T(\sigma)\circ w \circ
  \varphi_\T(\sigma)^{-1})\chi = {}^\sigma\!(w\cdot
  {}^{\sigma^{-1}}\!\chi) = \sigma(w)\cdot \chi$$ where we are using
  the natural Galois action on the Weyl group.  Therefore,
  $\varphi_\T(\sigma)\circ w \circ \varphi_\T(\sigma)^{-1}= \sigma(w)$
  which does indeed belong to $W(\G,\T)$.

  \noindent (ii) The isomorphism of Weyl groups is easy to check; if
  $w\in W(\G,\T)$ has representative $n\in N_\G(\T)(\kbar),$ then
  $F\circ w \circ F^{-1}$ belongs to $W(\G,\T_0)$ with representative
  $x^{-1} n x$.  To verify that we have an isomorphism
  $\Pi(\G,\T)\xrightarrow{\sim} \Pi(\G,\T_0)$, it suffices to show
  that $F\circ \varphi_\T(\sigma) \circ F^{-1}$ belongs to
  $\Pi(\G,\T_0)$ for all $\sigma \in \Gal(\kbar/k)$.  For $\chi \in
  X(\T_0)$,
\begin{equation} \label{E:phi compatible}
(F\circ \varphi_\T(\sigma) \circ F^{-1})\chi 
= {}^\sigma\!(\chi \circ f) \circ f^{-1}
= {}^\sigma\!\chi \circ ( {}^\sigma\!f \circ f^{-1})
= {}^\sigma\!\chi \circ ( f\circ {}^\sigma\!f^{-1})^{-1}.
\end{equation}
The automorphism $f\circ {}^\sigma\!f^{-1}$ of $\T_{\kbar}$ maps an
element $t\in \T(\kbar)$ to $x^{-1}\sigma(x)\, t \,
(x^{-1}\sigma(x))^{-1}$ which equals $w_{\sigma}\cdot t$ where
$w_{\sigma}$ is the element of $W(\G,\T)$ represented by
$x^{-1}\sigma(x) \in N_\G(\T)(\kbar)$ (indeed, since $\T$ and $\T_0$
are both defined over $k$, the element $x^{-1}\sigma(x)$ normalizes
$\T$).  From (\ref{E:phi compatible}), we deduce that $F\circ
\varphi_\T(\sigma) \circ F^{-1} = w_{\sigma} \circ
\varphi_{\T_0}(\sigma)$ which certainly belongs to $\Pi(\G,\T_0)$.  We
have also proved (iii).

\noindent For (iv), we may assume that $\T_0$ was chosen such that
$k_{\T_0}\subseteq K$.  For $\sigma \in \Gal(\kbar/K)$, part (iii)
implies that $\varphi_\T(\sigma)=F^{-1} \circ w_\sigma \circ F$ which
is an element of $W(\G,\T)$ by (ii).
\end{proof}

The groups $W(\G,\T)$ and $\Pi(\G,\T)$ are, up to isomorphism,
independent of $\T$ (by Proposition~\ref{P:Weyl group}(ii)).  We shall
denote the abstract groups simply by $W(\G)$ and $\Pi(\G)$,
respectively, when the choice of torus is unimportant.  The
isomorphisms $\Pi(\G,\T)\xrightarrow{\sim} \Pi(\G,\T_0)$ and
$W(\G,\T)\xrightarrow{\sim} W(\G,\T_0)$ of Proposition~\ref{P:Weyl
  group} are unique up to composition with an inner automorphism by an
element of the Weyl group; hence they induce \emph{canonical}
bijections $W(\G,\T)^\sharp = W(\G,\T_0)^\sharp$ and $\Pi(\G,\T)^\sharp
= \Pi(\G,\T_0)^\sharp$ of conjugacy classes. The set $W(\G)^\sharp$ and
$\Pi(\G)^\sharp$ are thus completely unambiguous.

We define \emph{the splitting field of $\G$} to be the field $k_\G :=
\bigcap_{\T} k_\T$ where the intersection is over all maximal tori
$\T$ of $\G$.  In other words, $k_\G$ is the largest extension of $k$
that is contained in any $K\subseteq \kbar$ for which $\G_K$ is split.

\begin{lemma} \label{L:in Weyl group} For every maximal torus $\T$ of
  $\G$, we have $\varphi_{\T}(\Gal(\kbar/k_{\G})) \subseteq W(\G)$.
  In particular, $\Gal(k_\T/k_\G)$ is isomorphic to a subgroup of
  $W(\G)$.
\end{lemma}
\begin{proof}
  Let $K\subseteq \kbar$ be the minimal extension of $k$ for which
  $\varphi_{\T}(\Gal(\kbar/K)) \subseteq W(\G,\T)$ (this is
  well-defined since $W(\G)$ is a normal subgroup of $\Pi(\G)$).  For
  a maximal torus $\T_0$ of $\G$, Proposition~\ref{P:Weyl group}(iv)
  implies that $K\subseteq k_{\T_0}$. Since $\T_0$ was arbitrary, we
  deduce that $K\subseteq k_{\G}$.
\end{proof}

\subsection{Galois groups for elements} \label{SS:Galois groups for
  elements}
  
Choose a faithful representation $\rho\colon \G\hookrightarrow \GL(m)$
defined over $k$.  For $g\in \G(k)$, we define $k_g$ to be the
splitting field over $k$ of $\det(T-\rho(g))$.

Recall that each $g\in \G(k)$ equals $g_s g_u$ for unique commuting
elements $g_s, g_u \in\G(k)$ where $g_s$ is semisimple and $g_u$ is
unipotent.  Since
$\det(T-\rho(g))=\det(T-\rho(g)_s)=\det(T-\rho(g_s))$, we have
$k_{g}=k_{g_s}$.  The \emph{unipotent radical} $R_u(\G)$ of $\G$ is
the maximal connected unipotent normal subgroup of $\G$.
The quotient $\G/R_u(\G)$ is reductive and defined over $k$.   

\begin{lemma}   \label{L:kg facts}
\begin{romanenum}
\item The field $k_g$ does not depend on the choice of $\rho$.
\item Define the reductive group $\G':=\G/R_u(\G)$ and let $\pi\colon
  \G \to \G'$ be the quotient homomorphism.  Then $k_g=k_{\pi(g)}$ for
  all $g\in \G(k)$.
\end{romanenum}
\end{lemma}
\begin{proof}
  Let $\mathbf{D}$ be the algebraic subgroup of $\G$ generated by
  $g_s$.  The group $\mathbf{D}$ is \emph{diagonalizable}, i.e.,
  $\mathbf{D}_{\kbar}$ is isomorphic to a subgroup of some torus
  $\GG^r_{m,\kbar}$.  Let $K$ be the \emph{splitting field} of
  $\mathbf{D}$, that is, the smallest extension $K\subseteq \kbar$ of
  $k$ for which $\mathbf{D}_K$ is isomorphic to a subgroup of a split
  torus $\GG^r_{m,K}$.  By \cite[\S8.4]{borel}, we find that $K$ is
  also the smallest extension of $k$ such that a $\GL(m,K)$-conjugate
  of $\rho(\mathbf{D}_K)$ is contained in the diagonal subgroup of
  $\GL(m)$.  Equivalently, $K$ is the smallest extension of $k$ for
  which $\rho(g_s)=\rho(g)_s$ is $\GL(m,K)$-conjugate to a diagonal
  matrix.  Therefore, $K=k_g$ and part (i) follows since our
  description of $K$ does not depend on $\rho$.

  Let $\mathbf{D}'$ be the algebraic subgroup of $\G'$ generated by
  $\pi(g)_s=\pi(g_s)$.  We have $\mathbf{D}\cap R_u(\G)=1$ since the
  only semisimple and unipotent element is 1.  Therefore,
  $\pi|_{\mathbf{D}}\colon \mathbf{D}\to \mathbf{D}'$ is an
  isomorphism of algebraic groups.  Since $\mathbf{D}$ and
  $\mathbf{D}'$ are isomorphic, we must have $k_g=k_{\pi(g)}$.
\end{proof}

Recall that a semisimple $g\in \G(k)$ is \emph{regular} in $\G$ if it
is contained in a unique maximal torus; we shall denote this maximal
torus by $\T_g$.  For a semisimple and regular $g\in \G(k)$, we define
\[
\varphi_g \colon \Gal(\kbar/k) \to \Pi(\G)
\]
to be the representation denoted by $\varphi_{\T_g}$ in the previous
section. The representation $\varphi_g$ is uniquely defined up to an
inner automorphism by an element of $W(\G)$.

We will now relate the fields $k_g$ to the Galois extensions arising
from maximal tori of $\G$.

\begin{lemma}\label{lm-tg-dg} Assume that $\G$ is reductive.
\begin{romanenum}   
\item For all $g\in \G(k)$, $\Gal(k_g/k)$ is isomorphic to a
  subquotient of $\Pi(\G)$ and $\Gal(k_\G k_g/k_\G)$ is isomorphic to
  a subquotient of $W(\G)$.
\item For $g\in \G(k)$, the field $k_g$ is the extension of $k$
  generated by the set $\{\chi(g_s): \chi\in X(\T)\}$ where $\T$ is a
  maximal torus of $\G$ containing $g_s$.
\item There is a closed subvariety $Y\subsetneq \G$ that is stable
  under conjugation by $\G$ such that if $g\in \G(k)-Y(k)$, then $g$
  is semisimple and regular in $\G$ and $k_g=k_{\T_g}$.
\end{romanenum}
\end{lemma}

\begin{proof}
  We start with (ii).  Take $g\in \G(k)$.  Since $k_g=k_{g_s}$, we may
  assume that $g$ is semisimple.  Fix a maximal torus $\T$ containing
  $g$, and let $\Omega\subseteq X(\T)$ be the set of \emph{weights}
  arising from the representation $\rho|_{\T}\colon \T \hookrightarrow
  \GL(m)$.  There are positive integers $m_\chi$ such that 
$$
\det(T-\rho(t))= \prod_{\chi\in \Omega}(T-\chi(t))^{m_\chi}
$$ 
for all $t\in \T(\kbar)$, and in particular, $\{\chi(g):\chi\in
\Omega\}$ is the set of roots of $\det(T-\rho(g))$ in $\kbar$.  The
set $\Omega$ generates the group $X(\T)$ since the representation
$\rho|_{\T}\colon \T\to \GL(m)$ is faithful.  Therefore, we see that
the extension of $k$ generated by $\{\chi(g):\chi\in X(\T)\}$ is equal
to $k_g=k(\{\chi(g):\chi\in \Omega\})$, which  completes the proof of
(ii).
\par
Now we prove (i). For $\sigma\in \Gal(\kbar/k)$, we have
\begin{equation}\label{E:weights}
  \sigma(\chi(g))={}^\sigma\!\chi (\sigma(g))={}^\sigma\!\chi (g)
\end{equation}
for all $\chi\in \Omega$.  In particular, $\sigma(\chi(g))=\chi(g)$
for all $\sigma\in \Gal(\kbar/k_\T)$.  Since $k_g$ is generated over
$k$ by $\{\chi(g): \chi \in \Omega\}$, we deduce that $k_\T \supseteq
k_g$.  Part (i) follows, since in \S\ref{galoistori} we saw that
$\Gal(k_\T/k)$ was isomorphic to a subgroup of $\Pi(\G)$ and
$\Gal(k_\T/k_\G)$ was isomorphic to a subgroup of $W(\G)$.
\par
Now for (iii). First of all, there is a closed subvariety
$Y_1\subseteq \G$ such that $h\in \G(\kbar)$ does not belong to
$Y_1(\kbar)$ if and only if $h$ is semisimple and regular in
$\G_{\kbar}$, see~\cite[2.14]{steinberg} (the proof is given there
only for semisimple groups, but the reductive case follows easily from
the semisimple case by considering the morphism from $\G$ to
$\G/R_u(\G)$).
\par
Now fix a maximal torus $\T_0\subset \G$, and let $\Omega_0$ be the
set of weights of $\T_0$ with respect to $\rho$, as above.  The set
$$
V=\{t\in\T_0\,\mid\, \text{ the $\chi(t)$ are distinct for }
\chi\in\Omega_0\}
$$
is an open dense subset of $\T_0$. Arguing as
in~\cite[2.14]{steinberg}, it follows that the set $Y_2$ of those
$h\in \G$ where $g_s$ is conjugate in $\G$ to an element in $\T_0-V$ is
a proper subvariety of $\G$. Moreover, it is clearly invariant under
conjugation.
\par
Now we define the proper closed subvariety $Y=Y_1\cup Y_2$ of $\G$,
which is stable under conjugation, and we claim that (iii)
holds. Indeed, let $g\in \G(k)-Y(k)$.  Since $g\notin Y_1(k)$, it is a
regular semisimple element of $\G$. Let $\T_g$ be the unique maximal
torus containing $g$, $\Omega$ the set of weights with respect to
$\T_g$. Since $g\notin Y_2(k)$ and $\Omega$ is obtained from
$\Omega_0$ by conjugation, it follows that the values $\chi(g)$,
$\chi\in\Omega$, are all distinct. 
\par
But now, take any $\sigma\in \Gal(\kbar/k_g)$.  By (\ref{E:weights}),
we have ${}^\sigma\!\chi (g)=\chi(g)$ for all $\chi\in \Omega$, and
therefore we must have in fact ${}^\sigma\!\chi=\chi$ for all $\chi\in
\Omega$.  Since $\Omega$ generates the group $X(\T_g)$, we find that
$\sigma$ acts trivially on $X(\T_g)$, and since $\sigma$ was an
arbitrary element of $\Gal(\kbar/k_g)$, we deduce finally that
$k_g\supseteq k_{\T_g}$.
\end{proof}

\section{Reductions of arithmetic groups and tori over finite
  fields} \label{SS:reduction}

Let $\G$ be a connected semisimple group defined over a number field
$k$.  To consider reductions, we will need to choose a model of $\G$.
This means that we take a group scheme $\mathcal{G}$ over a ring
$\Zz_k[R^{-1}]$ whose generic fiber $\mathcal{G}_k$ is isomorphic to
$\G$, where $R$ is a finite set of maximal ideals of $\Zz_k$.  We
identify $\G$ with the generic fiber of $\mathcal{G}$.  Any two such
models will agree after possibly inverting more primes.  From now on,
$\p$ will denote a maximal ideal of $\Zz_k$.  Let $k_{\p}$ be the
completion of $k$ at the prime $\p$ and let $\mathcal{O}_{\p}$ be the
corresponding valuation ring.  The ring $\mathcal{O}_{\p}$ is a
discrete valuation ring with residue field $\Fp_{\p}$.

After possibly increasing $R$, we may assume that $\mathcal{G}$ is
semisimple and that all of its fibers have the same type.  For
background on general reductive groups, see \cite{demazure}; recall
that $\mathcal{G}$ is \emph{semisimple} if it is affine and smooth
over $\Zz_k[R^{-1}]$ and if the generic fiber $\mathcal{G}_k$ and
special fibers $\mathcal{G}_{\Fp_{\p}}$ ($\p\notin R$) are semisimple
in the usual sense.

Choose a maximal torus $\calT_0$ of $\calG$.  Let $\calP$ be the set
of maximal ideals $\p \notin R$ of $\Zz_k$ such that the tori
$\calT_{0,k_\p}$ and $\calT_{0,\FF_\p}$ are both split.

\begin{lemma} \label{L:conjugacy bijection} Let $\mathcal{N}_0$ be the
  normalizer of $\mathcal{T}_0$ in $\mathcal{G}$.  For each $\p\in
  \calP$, there is a unique bijection
\begin{equation*} \label{E:conjugacy bijection}
W(\calG_{k_\p})^\sharp \leftrightarrow  W(\mathcal{G}_{\Fp_{\p}})^\sharp
\end{equation*}
such that for $n\in \mathcal{N}_0(\mathcal{O}_\p)$ the image of $n$ in
$W(\calG_{k_\p},\mathcal{T}_{0,k_\p})^\sharp$ and
$W(\calG_{\Fp_\p},\mathcal{T}_{0,\Fp_\p})^\sharp$ correspond.
\end{lemma}
\begin{proof}

  The homomorphism
  \begin{equation} \label{E:Weyl isom 1}
    \mathcal{N}_0(\mathcal{O}_{\p})/\mathcal{T}_0(\mathcal{O}_{\p})
    \hookrightarrow
    \mathcal{N}_0(k_{\p})/\mathcal{T}_0(k_{\p}) =
    W(\mathcal{G}_{k_{\p}},\mathcal{T}_{0,k_{\p}})
\end{equation}
is injective; the identification with the Weyl group uses that
$\mathcal{T}_{0,k_{\p}}$ is split.  The normalizer
$\mathcal{N}_0$ is a closed and smooth subscheme of
$\mathcal{G}$ (for smoothness, cf.~\cite[XXII~Corollaire
5.3.10]{SGA3}).  The homomorphisms
$\mathcal{N}_0(\mathcal{O}_{\p})\to
\mathcal{N}_0(\Fp_{\p})$ and
$\mathcal{T}_0(\mathcal{O}_{\p})\to \mathcal{T}_0(\Fp_{\p})$ are
surjective by Hensel's lemma.  We thus have a surjective homomorphism
\begin{equation} \label{E:Weyl isom 2}
  \mathcal{N}_0(\mathcal{O}_{\p})/\mathcal{T}_0(\mathcal{O}_{\p})
  \twoheadrightarrow
  \mathcal{N}_0(\Fp_{\p})/\mathcal{T}_{0,\Fp_{\p}}(\Fp_{\p})
  = W(\mathcal{G}_{\Fp_{\p}},\mathcal{T}_{0,\Fp_{\p}})
\end{equation}
where the equality uses that $\mathcal{T}_{0,\Fp_{\p}}$ is split.  The
Weyl groups $W(\mathcal{G}_{k_{\p}},\mathcal{T}_{0,k_{\p}})$ and
$W(\mathcal{G}_{\Fp_{\p}},\mathcal{T}_{0,\Fp_{\p}})$ are isomorphic
since $\mathcal{G}_{k_{\p}}$ and $\mathcal{G}_{\Fp_{\p}}$ have the
same type.  Since (\ref{E:Weyl isom 1}) and (\ref{E:Weyl isom 2}) are
injective and surjective, respectively, we deduce that they are both
isomorphisms. By combining them, we get an isomorphism
$$
W(\mathcal{G}_{k_{\p}},\mathcal{T}_{0,k_{\p}}) \xrightarrow{\sim}
W(\mathcal{G}_{\Fp_{\p}},\mathcal{T}_{0,\Fp_{\p}}).
$$
\par 
The desired bijection of conjugacy classes is induced from this
isomorphism.  The uniqueness is a consequence of the surjectivity of
(\ref{E:Weyl isom 2}).
\end{proof}

Fix a maximal ideal $\p \in \calP$ and choose an embedding
$\iota\colon \kbar\hookrightarrow \overline{k_\p}$ that is the
identity on $k$.  Using $\iota$, we can make an identification
$W(\G,\mathcal{T}_{0,k}) = W(\calG_{k_\p},\mathcal{T}_{0,k_\p})$.
Combining with the map of Lemma~\ref{L:conjugacy bijection}, we obtain
a bijection
\begin{equation} \label{E:conjugacy bijection 2}
W(\G)^\sharp \leftrightarrow  W(\mathcal{G}_{\Fp_{\p}})^\sharp
\end{equation}
that we will also use as an identification.  For an element
$g\in\calG(\FF_\p)$ that is semisimple and regular in
$\calG_{\FF_\p}$, we have a homomorphism
\[
\varphi_{g} \colon \Gal({\Fbar}_\p/\Fp_\p) \to W(\G)
\]
by using that $\calT_{0,\FF_\p}$ is split.  Let $\Frob_\p$ be the
Frobenius automorphism $x\mapsto x^{N(\p)}$ of $\Fbar_\p$ where
$N(\p)$ is the cardinality of $\FF_\p$.  The representation
$\varphi_g$ (up to inner automorphism) is determined by the conjugacy
class $\varphi_g(\Frob_\p)$ of $W(\G)$.
\par
The following crucial proposition shows that the local and global
images of Frobenius automorphisms coincide.

\begin{proposition}[Local and global Frobenius]\label{P:Frobenius
    comparison 1}
  Let $\p\in \calP$ be a prime ideal.  Let $g\in
  \mathcal{G}(\Zz_k[R^{-1}])$ be an element such that $g$ is
  semisimple and regular in $\G=\mathcal{G}_{k}$ and $\bar{g}:= g
  \bmod{\p} \in \mathcal{G}(\Fp_\p)$ is semisimple and regular in
  $\calG_{\Fp_\p}$.  Let $C$ be the conjugacy class of $W(\G)$
  containing $\varphi_{\bar{g}}(\Frob_{\p})$.  Then the representation
  $\varphi_g$ is unramified at $\p$, and if $\sigma_{\p}\in
  \Gal(\bar{k}/k)$ denotes a Frobenius element at $\p$, we have
$$
\varphi_g(\sigma_\p) \in C.
$$
\end{proposition}

This is intuitively very natural, but the proof requires some care in
our generality, in particular to show that the representation is
unramified. In previous works, this issue did not come up, since one
could explicitly control factorizations of the reduction of the
characteristic polynomial to ensure it was squarefree in suitable
conditions.
\par
Because of this proposition, we will, from now on, also denote by
$\Frob_{\p}$ any representative of the Frobenius automorphism in
$\Gal(\bar{k}/k)$. 

\begin{proof}
  Let $k_{\p}^{un}$ denote the maximal unramified extension of
  $k_{\p}$ in an algebraic closure $\overline{k_\p}$.  Let
  $\mathcal{O}_{\p}^{un}$ be the valuation ring of $k_\p^{un}$; its
  residue field is $\Fbar_{\p}$.  We have an isomorphism
$$
\Gal(k_\p^{un}/k_\p)\xrightarrow{\sim}\Gal(\overline{\Fp}_\p/\Fp_\p),
$$
which allows us to view $\Frob_\p$ as an automorphism of
$k_\p^{un}$. We will also denote by $\Frob_\p$ any extension of the
Frobenius automorphism to the field $\overline{k_\p}$

We now need to compare maximal tori in $\cG_{\cO_{\p}}$ and
$\cG_{\Fp_{\p}}$. First, we have the tori $\cT_{0,\cO_{\p}}$ (which we
will still denote $\cT_0$ for simplicity) and its reduction
$\cT_{0,\Fp_{\p}}$ modulo $\p$. Further, because we assume that $g$
(as element of $\cG(k_{\p})$, i.e., of the generic fiber of
$\cG_{\cO_{\p}}$) and $\bar{g}$ (as section of $\cG_{\cO_{\p}}$ over
the special fiber) are regular semisimple, there exists a unique
maximal torus $\cT$ of $\cG_{\cO_{\p}}$ containing $g$ (this is a
special case of~\cite[XIII, Cor. 3.2]{SGA3}).
\par
The transporter $\textrm{Transp}_{\mathcal{G}}(\mathcal{T}_0,
\mathcal{T})$, defined as an $\cO_{\p}$-scheme by
$$
\textrm{Transp}_{\mathcal{G}}(\mathcal{T}_0, \mathcal{T})(A) = \{ g
\in \mathcal{G}(A) : g\mathcal{T}_{0,A} g^{-1} = \mathcal{T}_{A} \},
$$
is a closed and smooth group scheme in $\mathcal{G}_{\mathcal{O}_\p}$
(for smoothness, cf.~\cite[XXII~Corollaire 5.3.10]{SGA3}).  
\par
Now we can choose an element $\overline{x} \in
\textrm{Transp}_{\mathcal{G}}(\mathcal{T}_0,
\mathcal{T})(\Fbar_{\p})$.  Since
$\textrm{Transp}_{\mathcal{G}}(\mathcal{T}_0, \mathcal{T})$ is smooth
and $\mathcal{O}_{\p}^{un}$ is a Henselian ring, there is an $x \in
\textrm{Transp}_{\mathcal{G}}(\mathcal{T}_0,
\mathcal{T})(\mathcal{O}_{\p}^{un})$ which lifts $\overline{x}$.
\par
Finally, by Proposition~\ref{P:Weyl group}(iii), the conjugacy class
of $\varphi_{\mathcal{T}_{k_{\p}}}(\Frob_{\p})$ in
$W(\mathcal{G}_{k_{\p}})^\sharp=W(\mathcal{G}_{k_{\p}},
\mathcal{T}_{0,k_{\p}})^\sharp$ is represented by $x^{-1}
\Frob_{\p}(x) \in
N_{\mathcal{G}_{k_{\p}}}(\mathcal{T}_{0,k_{\p}})$. Similarly, the
conjugacy class of
$\varphi_{\bar{g}}(\Frob_\p)=\varphi_{\mathcal{T}_{\Fp_{\p}}}(\Frob_{\p})$
in $W(\mathcal{G}_{\Fp_{\p}})^\sharp=W(\mathcal{G}_{\Fp_{\p}},
\mathcal{T}_{0,\Fp_{\p}})^\sharp$ is represented by $\overline{x}^{-1}
\Frob_{\p}(\overline{x}) \in
N_{\mathcal{G}_{\Fp_{\p}}}(\mathcal{T}_{0,\Fp_{\p}})$.  These
conjugacy correspond under the bijection of Lemma~\ref{L:conjugacy
  bijection}.  In particular,
$\varphi_{\mathcal{T}_{k_{\p}}}(\Frob_{\p})$ does not depend on the
choice of extension $\Frob_\p$; i.e., $\varphi_{\mathcal{T}_{k_{\p}}}$
is unramified at $\p$.  The choice of $\iota$ gives an inclusion
$\iota^*\colon \Gal(\overline{k_\p}/k_\p) \hookrightarrow
\Gal(\kbar/k)$.  The representation $\varphi_{\calT_{k_\p}}$ is equal
to $\iota^*$ composed with $\varphi_{g}$.  The proposition now follows
immediately.
\end{proof}

To show that $\varphi_{g}$ is often surjective, we will exploit the
well-known lemma of Jordan according to which, in a finite group, no
proper subgroup contains elements of all conjugacy classes.
Proposition~\ref{P:Frobenius comparison 1} will be used to produce
conjugacy classes of $W(\G)$ that intersect
$\varphi_g(\Gal(\kbar/k_\G))$.

\section{Semisimple elements and conjugacy classes in the Weyl group
  over finite fields} \label{SS:finite fields}

In this section, which can be read independently of the rest of the
paper, we consider a finite field $k=\Fp_{q}$ with $q$ elements and a
connected split semisimple group $\G$ defined over $\Fp_q$.  In
\S\ref{sec-tori}, we defined a homomorphism
\[
\varphi_{\T} \colon \Gal(\Fbar_q/\Fp_q) \to W(\G,\T)
\]
for each maximal torus $\T$ of $\G$ (the image lies in $W(\G,\T)$ by
Proposition~\ref{P:Weyl group}(iv)).  The representation
$\varphi_{\T}$ is determined by its value on the Frobenius
automorphism $F\colon x\mapsto x^{q}$, and $\varphi_{\T}(F)$ gives a
well-defined conjugacy class in $W(\G,\T)^\sharp=W(\G)^\sharp$ that we
shall denote by $\theta(\T)$.  
 
Now if $g\in \G(\Fp_{q})$ is a semisimple regular element of $\G$,
then it is contained in a unique maximal torus $\T_g$ of $\G$, and we
will study here the map
\begin{align*}
\theta \colon \left\{
\begin{array}{ccl}
\G(\Fp_{q})_{sr} & \to& W(\G)^\sharp,\\
 g& \mapsto& \theta(\T_g)=[\varphi_{\T_g}(F)]
\end{array}
\right.
\end{align*}
where $\G(\Fp_{q})_{sr}$ is the set of regular and semisimple elements
of $\G(\Fp_{q})$. From Proposition~\ref{P:Weyl group}(iii), it
follows that this can be described concretely as follows: we fix a split
maximal torus $\T_0\subset \G$, and then, given $g\in \G(\Fp_q)_{sr}$, let
$\T$ be the unique maximal torus containing $g$.  Take $y\in \G(\Fbar_q)$ such
that
$$
\T=y\T_0 y^{-1}.
$$
Then $\theta(g)$ is the class of $y^{-1}F(y)$ in $W(\G,\T_0)^{\sharp}=
W(\G)^{\sharp}$.
\par
Our goal is to prove that the values of this map are asymptotically
equidistributed, with respect to the natural measure on the conjugacy
classes of $W(\G)$, when $q$ goes to infinity (and the type of $\G$ is
fixed). 

\begin{proposition} \label{P:Carter approximation}
For each $C\in W(\G)^\sharp$, we have
\[
\frac{\big|\big\{g \in \G(\Fp_{q})_{sr}: \theta(g)=C\big\}\big|}
{|\G(\Fp_{q})|}= \frac{|C|}{|W(\G)|} (1 + O(q^{-1}))
\]
where the implicit constant depends only on the type of $\G$.
\end{proposition}

\begin{remark}
  If $\G$ is a simple and simply-connected group, Theorem~1
  of~\cite{carter} describes precisely the number of semisimple
  conjugacy classes of $\G(\Fp_q)$ mapping to $C$ under $\theta$ in
  terms of the geometry of the action of the so-called affine Weyl
  group on the cocharacter group of a maximal torus of $\G$. A proof
  of Proposition~\ref{P:Carter approximation} can then be derived
  fairly easily by a lattice-point counting technique, the well-known
  formula for the volume of the fundamental domain of the affine Weyl
  group, and some equidistribution of semisimple conjugacy classes.
  Our proof is different; it requires less precise information (and
  works for arbitrary connected semisimple groups), exploiting the fact
  that we only look for asymptotic information for large $q$.
\end{remark}

\begin{remark}
  This map has already been considered by Fulman~\cite{fulman} and
  Carter~\cite{carter-3} in the context of finite groups of Lie type.
  As they remark, it takes a very classical and concrete form when
  $\G=\SL(m)$. In that case, the Weyl group is the symmetric group
  $\mathfrak{S}_m$, and its conjugacy classes correspond naturally to
  partitions of the integer $m$. Now, consider an element $g\in
  \SL(m,\FF_q)$ which is regular and has distinct eigenvalues in
  $\Fbar_q$; in that case its characteristic polynomial
  $\det(T-g)\in\FF_q[T]$ is monic, squarefree and of degree $m$. We
  may factor it as a product of distinct irreducible factors
$$
\det(T-g)=\pi_1\cdots \pi_k
$$
and the degrees $d_i=\deg(\pi_i)$ form a partition $\lambda$ of $m$
(with as many cycles of length $j$, for $1\leq j\leq m$, as there are
factors of degree $d_i$ equal to $j$); then one can check that
$\theta(g)$ is the conjugacy class in $\mathfrak{S}_m$ corresponding
precisely to this partition. 
\end{remark}

\subsection{Proof of Proposition~\ref{P:Carter approximation}}

We will first need a few lemmas. The notation in this section is the
same as before.

\begin{lemma} \label{L:theta correspondence} 
The map $\T\mapsto
  \theta(\T)$ defines a bijection between the maximal tori of $\G$ up
  to conjugation by $\G(\Fp_q)$ and the conjugacy classes
  $W(\G)^\sharp$.
\end{lemma}

\begin{proof}
  This is \cite[Prop. 3.3.3]{carter} though stated a little
  differently.  First, fix a split maximal torus $\T_0$ of $\G$.  The
  action of $F$ on $W(\G,\T_0)$ is trivial since $\T_0$ is split, so
  $F$-conjugacy classes of $W(\G,\T_0)$ in \cite{carter} are the same
  as usual conjugacy classes.  The equivalence of our statement and
  Carter's then follows using $W(\G,\T)^\sharp=W(\G,\T_0)^\sharp$ and
  Proposition~\ref{P:Weyl group}(iii).
\end{proof}

We also recall that for any connected reductive group $\G/\Fp_q$, we
have
\begin{equation}\label{eq:card}
(q-1)^{\dim \G}\leq |\G(\Fp_q)|\leq (q+1)^{\dim\G},
\end{equation}
as follows from the formula of Steinberg for $|\G(\Fp_q)|$ (see,
e.g.,~\cite[p. 75, Prop. 3.3.5]{carter}).
\par
The next lemma is well-known, and essentially follows from Lang-Weil
estimates in our application, but since we think of this as a fact
about finite groups of Lie type, and not about reductions of groups
over numbers fields, we give the details (the argument is, in any
case, more elementary than the use of the Lang-Weil bounds).

\begin{lemma} \label{L:strongly regular bound} 
With notation as above, we have
$$
|\G(\Fp_{q})_{sr}| = |\G(\Fp_{q})|(1+O(q^{-1})),
$$ 
where the implicit constant depends only on the type of $\G$. 
\end{lemma}

\begin{proof}
  As already observed, any element in $\G(\Fp_{q})_{sr}$ lies in a
  unique maximal torus of $\G$.  Hence we have
\begin{equation} \label{E:into pieces}
|\G(\Fp_{q})_{sr}|=\sum_{\Tt \in \mathcal{T}}{|\Tt(\Fp_q)\cap \G(\Fp_q)_{sr}|}
\end{equation}
where $\mathcal{T}$ is the set of ($F$-stable) maximal tori in
$\G$. 
\par
Now fix a maximal torus $\Tt\in\mathcal{T}$, and let
$\Phi=\Phi(\G,\T)$ be the set of roots of $\G$ with respect to $\Tt$.
Fix an element $x\in \T(\FF_q)$ that is not regular in $\G$.  Then
there exists a root $\alpha \in \Phi$ such that $\alpha(x)=1$
\cite[III.12.2]{borel}.  So let $A$ be the (non-empty) set of roots
$\alpha\in \Phi$ for which $\alpha(x)=1$, and define the algebraic
subgroup
$$
\D_{A}=\bigcap_{\alpha\in A}{\ker\alpha}
$$
of $\T$.  Since $A$ is $F$-stable, the group $\D_A$ is defined over
$\Fp_q$ and $x\in \D_A(\Fp_q)$.  We thus have
\[
|\{x\in \T(\Fp_q) : x \text{ is not regular in $\G$}\}| \leq \sum_{A}
|\D_A(\Fp_q)|,
\]
where the sum is over all non-empty $F$-stable subsets $A\subseteq \Phi$.
For any such subset $A\subseteq \Phi$, we claim that $|\D_A(\Fp_q)|=
O(q^{r-1})$ where $r$ is the dimension of $\T$ and the implied
constant depends only on the type of $\G$.  Assuming this for now, we
have
\[
|\{x\in \T(\Fp_q) : x \text{ is not regular in $\G$}\}| \ll
 \sum_A q^{r-1} \ll q^{r-1}
\]
where the implied constant again depends only on the type of $\G$, and
hence $|\T(\Fp_q) \cap \G(\Fp_q)_{sr}| = |\T(\Fp_q)| + O(q^{r-1})$.
Applying~(\ref{eq:card}) to $\T$, we get
\[
|\Tt(\Fp_q)\cap \G(\Fp_q)_{sr}| = q^r + O(q^{r-1}).
\]

We now return to (\ref{E:into pieces}).  According to a theorem of
Steinberg~\cite[Th. 3.4.1]{carter}, we have $|\mathcal{T}|=q^{2N}$
where $N$ is the number of positive roots of $\G$, so
\[
|\G(\Fp_{q})_{sr}|=\sum_{\Tt \in \mathcal{T}} |\Tt(\Fp_q)\cap
\G(\Fp_q)_{sr}| = |\mathcal{T}| \big(q^r + O(q^{r-1})\big) = q^{2N+r}
+ O(q^{2N+r-1})
\]
where the implied constant depends only on the type of $\G$.  The
desired estimate for $|\G(\Fp_{q})_{sr}|$ follows by noting that
$2N+r=\dim \G$ and applying~(\ref{eq:card}) to $\G$.

It remains to show that for a fixed maximal torus $\T$ and a non-empty
$F$-stable set $A$ of roots of $\G$ relative to $\T$, we have
$|\D_A(\Fp_q)|= O(q^{r-1})$ where the implied constant depends only on
the type of $\G$.  Since the connected component of the identity of a
diagonalizable group is a torus such that
$$
|\D_A^0(\Fp_q)|\leq (q+1)^{\dim \D_A}\leq (q+1)^{r-1}
$$
it is enough to show that the number of (geometric) connected
components of $\D_A$ is bounded in terms of the type of $\G$ only
(note that $\dim \D_A<\dim \Tt$, since $A$ is non-empty).  From the
exact sequence
$$
1\ra \ker(\alpha)\ra \Tt\fleche{\alpha}\G_m\ra 1,
$$
for $\alpha\in \Phi$, and the dual exact sequence
$$
0\ra \Zz\ra X(\Tt)\ra X(\ker(\alpha))\ra 0
$$
of abelian groups of finite rank (see~\cite[III.8.12]{borel}), we find
that the character group of $X(\D_A)$ is
\begin{equation}\label{eq-proj}
X(\D_A)\simeq X(\Tt)/\langle \Zz\alpha\,\mid\, \alpha\in A\rangle.
\end{equation}
\par
The fundamental structure theory of reductive groups shows that the
subgroup $\Psi$ of $X(\Tt)$ generated by the roots $\Phi$ together
with a basis of the characters of the center of $\G$ is of bounded
index in $X(\Tt)$, the bound depending only on the type of $\G$ (see,
e.g.,~\cite[1.11]{carter}). Thus the size of the torsion subgroup of
$X(\D_A)$ differs from that of
$$
\Psi_A=\Psi/\langle \Zz\alpha\,\mid\, \alpha\in A\rangle,
$$
only by a bound depending only on the type of $\G$. Moreover, $\Psi_A$
is defined purely in terms of the root datum, and therefore only
depends on the type of $\G$. Thus, the result follows.
\end{proof}

\begin{proof}[Proof of Proposition~\ref{P:Carter approximation}]
  Fix a conjugacy class $C\in W(\G)^\sharp$ and let $\mathscr{T}_C$ be
  the set of maximal tori $\T$ of $\G$ for which $\theta(\T)=C.$ Since
  a regular semisimple element of $\G$ lies in a unique maximal torus,
  we have
\begin{align*}
  \frac{|\{ g \in \G(\Fp_q)_{sr} : \theta(g)=C\}|}{| \G(\Fp_q)|}
  &= \sum_{\T \in \mathscr{T}_C}  
\frac{|  \T(\Fp_q) \cap \G(\Fp_q)_{sr}|}{|\G(\Fp_q)|}  \\
  & = \sum_{\T \in \mathscr{T}_C} \frac{| \T(\Fp_q)|}{|\G(\Fp_q)|} +
  O\Bigl(\frac{| \G(\Fp_q)- \G(\Fp_q)_{sr}|}{|\G(\Fp_q)|}\Bigr)\\
  & = \sum_{\T \in \mathscr{T}_C} \frac{| \T(\Fp_q)|}{|\G(\Fp_q)|} +
  O(q^{-1})
\end{align*}
where the last line uses Lemma~\ref{L:strongly regular bound} and the
implicit constant depends only on the type of $\G$.  It thus suffices
to show that
$$\frac{1}{|\G(\Fp_q)|}
\sum_{\T \in  \mathscr{T}_C}  |\T(\Fp_q)| = \frac{ |C|}{|W(\G)|}.
$$

By Lemma~\ref{L:theta correspondence}, any two tori in $\mathscr{T}_C$
are $\G(\Fp_q)$-conjugate.  So after fixing a $\T_1 \in \mathscr{T}_C$
(that $\mathscr{T}_C\neq \emptyset$ is part of Lemma~\ref{L:theta
  correspondence}), we have
\[
|\mathscr{T}_{C}|=\frac{|\G(\Fp_q)|}{|N_{\G}(\T_1)(\Fp_q)|}
\]
(the denominator being the order of the stabilizer of $\T_1$ under
$\G(\Fp_q)$-conjugation) and hence
\[
\frac{1}{|\G(\Fp_q)|} \sum_{\Tt \in \mathscr{T}_C}|\Tt(\Fp_q)| = 
\frac{1}{|\G(\Fp_q)|}|\mathscr{T}_C| |\T_1(\Fp_q)| 
=  \frac{|\T_1(\Fp_q)|}{|N_{\G}(\T_1)(\Fp_q)|}.
\] 
By Proposition~3.3.6 of \cite{carter}, we have
$|N_{\G}(\Tt_1)(\Fp_q)/\Tt_1(\Fp_q) |=| C_{W(\G,\T_0)}(w)|$ where
$\T_0$ is a split maximal torus of $\G$, $w\in W(\G,\T_0)$ lies in
$C\in W(\G,\T_0)^\sharp=W(\G)^\sharp$, and $C_{W(\G,\T_0)}(w)$ is the
centralizer of $w$ in $W(\G,\T_0)$ (the action of $F$ on $W(\G,\T_0)$
is trivial since $\T_0$ is split, so the $F$-centralizers in
\cite{carter} are the same as standard centralizers).  Since $|W(\G)|=
|C|\cdot |C_{W(\G,\T_0)}(w)|$, the desired formula follows.
\end{proof}

It will be important for our application to have uniform bounds and
have estimates for those elements lying in certain special cosets in
$\G(\Fp_q)$.  
Let $\varphi\colon \G^{sc} \to \G$ be the universal cover of $\G$ (as
an algebraic group), the group and morphism are also defined over
$\Fp_q$.  The semisimple group $\G^{sc}$ is simply connected and the
kernel $\pi_1$ of $\varphi$ is a finite group scheme contained in the
center of $\G^{sc}$.  Our refined equidistribution result is the
following.

\begin{proposition} \label{P:Carter approximation2} Let $\G$ be a
  split semisimple group over $\Fp_q$.  Let $\kappa$ be a coset of
  $\varphi(\G^{sc}(\Fp_q))$ in $\G(\Fp_q)$.  Then for each $C\in
  W(\G)^\sharp$, we have
\[
\frac{\big|\big\{g \in \kappa \cap \G(\Fp_q)_{sr}: 
\theta(g)=C\big\}\big|}{|\kappa|}= \frac{|C|}{|W(\G)|} (1 + O(q^{-1}))
\]
where the implicit constant depends only on the type of $\G$.
\end{proposition}

We start with another simple lemma.

\begin{lemma} \label{L:index inequality} Let $\kappa$ be a coset of
  $\varphi(\G^{sc}(\Fp_q))$ in $\G(\Fp_q)$.  Then for any maximal
  torus $\T$ of $\G$, we have
\[
\frac{|\T(\Fp_q)\cap \kappa|}{|\kappa|} =
\frac{{|\T(\Fp_q)|}}{|\G(\Fp_q)|}.  
\]
\end{lemma}
\begin{proof}
  The short exact sequence $1\to \pi_1 \to \G^{sc}
  \xrightarrow{\varphi} \G \to 1$ gives the following long exact
  sequence in Galois cohomology,
\[
1 \to \pi_1(\Fp_q) \to \G^{sc}(\Fp_q) \xrightarrow{\varphi} \G(\Fp_q)
\xrightarrow{\delta} H^1(\Fp_q,\pi_1) \to 1,
\]
since $H^1(\Fp_q,\G^{sc})=1$ by Steinberg's theorem
\cite[1.9]{steinberg}.  Thus there exists an element $\kappa_0 \in
H^1(\Fp_q,\pi_1)$ such that $g\in \G(\Fp_q)$ lies in $\kappa$ if and
only if $\delta(g) = \kappa_0$.  Since $\pi_1$ is contained in the
center of $\G^{sc}$, there is a maximal torus $\T^{sc}$ of $\G^{sc}$
giving an exact sequence $1\to \pi_1 \to \T^{sc} \xrightarrow{\varphi}
\T \to 1$ and a long exact sequence
\[
1 \to \pi_1(\Fp_q) \to \T^{sc}(\Fp_q) \xrightarrow{\varphi} \T(\Fp_q)
\xrightarrow{\delta'} H^1(\Fp_q,\pi_1) \to 1.
\]
\par
The homomorphism $\delta'$ agrees with the homomorphism $\delta$ when
restricted to $\T(\Fp_q)$.  Therefore,
\begin{align*}
  \frac{|\T(\Fp_q)\cap \kappa|}{|\T(\Fp_q)|} &= \frac{|\{ t\in
    \T(\Fp_q): \delta'(t)=\kappa_0\}|}{|\T(\Fp_q)|} =
  \frac{1}{|H^1(\Fp_q,\pi_1)|}
\end{align*}
and
\[ 
\frac{| \kappa|}{|\G(\Fp_q)|} = \frac{|\{ g\in \G(\Fp_q):
  \delta(t)=\kappa_0\}|}{|\G(\Fp_q)|} = \frac{1}{|H^1(\Fp_q,\pi_1)|}.
\]
\end{proof}

\begin{proof}[Proof of Proposition~\ref{P:Carter approximation2}]
  Fix a conjugacy class $C\in W(\G)^\sharp$ and let $\mathscr{T}_C$ be
  the set of maximal tori $\T$ of $\G$ for which $\theta(\T)=C.$ Since
  a regular semisimple element of $\G$ lies in a unique maximal torus,
  we have
\begin{align*}
  \frac{|\{ g \in \kappa \cap \G(\Fp_q)_{sr} : \theta(g)=C\}|}{|
    \kappa|}
  &= \sum_{\T \in \mathscr{T}_C}  \frac{|  \kappa \cap \T(\Fp_q) \cap \G(\Fp_q)_{sr}|}{| \kappa|}  \\
  & = \sum_{\T \in \mathscr{T}_C} \frac{| \kappa \cap \T(\Fp_q)|}{|
    \kappa|}
  + O\Bigl(\frac{| \G(\Fp_q)- \G(\Fp_q)_{sr}|}{| \kappa|}\Bigr)\\
  & = \sum_{\T \in \mathscr{T}_C} \frac{| \kappa \cap \T(\Fp_q)|}{|
    \kappa|} + O(q^{-1})
\end{align*}
where the last line uses Lemma~\ref{L:strongly regular bound} and
$|\G(\Fp_q)|/|\kappa| = O(1)$ (the implicit constants depend only on
the type of $\G$). By Lemma~\ref{L:index inequality}, we have
\[
\frac{|\{ g \in \kappa \cap \G(\Fp_q)_{sr} : \theta(g)=C\}|}{|
  \kappa|} = |\G(\Fp_q)|^{-1} \sum_{\T \in \mathscr{T}_C} |\T(\Fp_q)|
+ O(q^{-1}).
\]
This completes the proof since we have already proved that
$$
|\G(\Fp_q)|^{-1}\sum_{\T \in \mathscr{T}_C} |\T(\Fp_q)| =
|C|/|W(\G)|
$$ 
in the course of the proof of Proposition~\ref{P:Carter
  approximation2}.
\end{proof}

\section{Sieve for random walks on semisimple algebraic groups}
\label{sec-sieve}
To prove our main results in the next section, we will use sieve
methods. We first consider in this section the problem of obtaining a
general (upper-bound) sieve result for ``random'' elements of an
arithmetic group in a semisimple group over a number field.
\par
To give a meaning to ``random'' elements in $\G$, we use random walks,
as in~\cite[Ch. 7]{lsieve} (but see Section~\ref{sec-comments} for
comments on other possibilities). This involves a fair amount of
notation, but is otherwise quite convenient.
\par
In all this section, we therefore consider to have fixed the following
data:
\begin{itemize}
\item A number field $k$;
\item A connected semisimple algebraic group $\Gg/k$ (not necessarily
  split);
\item An arithmetic subgroup $\Gamma\subset \G(k)$ of $\G$, as defined
  in the introduction, e.g.,
$$
\Gamma=\rho(\Gg(k))\cap \GL(N,\Zz_k)
$$ 
for some faithful representation $\rho\,:\, \G\hookrightarrow \GL(N)$
over $k$.   
\item A finite symmetric (i.e., $s\in S$ implies $s^{-1}\in S$)
  generating $S$ set of $\Gamma$ (the group $\Gamma$ is finitely
  generated\footnote{\ In fact, finitely presented, which is a quite
    deeper property which we do not need.} by a theorem of Borel, see
  e.g.~\cite[Th. 4.17 (2)]{platonov-rapinchuk}); we will always assume
  that the pair $(\Gamma,S)$ is \emph{balanced}, by which we mean that
  either (i) $1\in S$, or (ii) there exists no non-trivial
  homomorphism $\Gamma\ra \Zz/2\Zz$ (this is in order to avoid
  possible issues with bipartite Cayley graphs, see~\cite[\S
  7.4]{lsieve} for a discussion of this point; we thank the referee
  for having reminded us of this issue);
\item A sequence $(\xi_n)$ of independent, identically distributed,
  random variables, defined on some probability space
  $(\Omega,\Sigma,\proba)$, taking values in $S$:
$$
\xi_n\,:\, \Omega\ra S,
$$
such that $p(s)=\proba(\xi_n=s)>0$ for all $s\in S$, and
$p(s)=p(s^{-1})$ for all $s$.
\end{itemize}

\begin{example}
  Readers not familiar with the general theory may take:
\begin{itemize}
\item The field $k=\Qq$; 
\item The group $\Gg=\SL(N)$, $N\geq 2$, with $\rho$ the inclusion in
  $\GL(N)$;
\item The arithmetic group $\Gamma=\SL(N,\Zz)$;
\item The system of generators $S$ of elementary matrices
  $\mathrm{Id}\pm E_{i,j}$ for distinct $i,j \in \{1,\ldots, n\}$
  where $E_{i,j}$ has zero in all entries except for the $(i,j)$-th
  where it is one; for $N=2$, we add $1$ to $S$ in order for
  $(\Gamma,S)$ to be balanced (when $N\geq 3$, there is no non-trivial
  map $\SL(N,\Zz)\ra \Zz/2\Zz$);
\item The probability space $\Omega=\{(s_n)_{n\geq 1}\,\mid\, s_n\in
  S\}$, with the product uniform normalized counting measure,
  $\xi_n(\omega)=s_n$ for $\omega=(s_n)_{n\geq 1}\in \Omega$, so that
  $p(s)=1/|S|$ for all $s$.
\end{itemize}
\par
This is a setting already considered in~\cite[\S 7]{lsieve}. Note
however that in that case $\Gg$ is simply connected, so much of the
work needed below to deal with the general case is unnecessary. For a
non-simply connected example, one may take $\Gg=\SO(N,N)$ for $N\geq
2$, and $\Gamma=\SO(N,N)(\Zz)$.
\end{example}

To have a meaningful asymptotic problem, the discrete group $\Gamma$
must be ``big enough''. It seems that the right way to quantify this
in our setting is simply to assume that $\Gamma$ is
\emph{Zariski-dense} in $\Gg$.  By the Borel Density Theorem (see,
e.g.,~\cite[Th. 4.10]{platonov-rapinchuk}), this assumption on the
arithmetic group $\Gamma$ can be formulated purely in terms of the
semisimple group $\G$: it means that for any simple component of
$\Gg$, say $\Hh$, and any real or complex completion $K$ of $k$, the
group $\Hh(K)$ is noncompact for the real or complex topology. In
particular, this holds whenever $\G$ is split.
\par
We then define
$$
X_0=1\in \Gamma,\quad\quad X_{n+1}=X_n\xi_{n+1},
$$
so that the sequence $(X_n)$ is a random walk on $\Gamma$.
\par
To perform the sieve, we require independence properties of reductions
modulo primes of arithmetic groups. This independence is only valid
for \emph{simply connected} groups, and to reduce to this case we use
ideas already found in~\cite{jouve} with some new tools. 
\par
Let $\varphi\,:\, \G^{sc}\ra \G$ be the simply connected covering of
$\G$ (as an algebraic group). Both $\G^{sc}$ and $\varphi$ are defined over
$k$, so we can define
$$
\Gamma^{sc}=\varphi(\varphi^{-1}(\Gamma)\cap \G^{sc}(k))\subset \Gamma.
$$
\par
It follows from basic facts about arithmetic groups (see,
e.g.,~\cite[Theorem 4.1]{platonov-rapinchuk}) that $\Gamma^{sc}$ is an
arithmetic subgroup of $\G$. In fact, since $\Gamma^{sc}\subset
\Gamma$, it follows that $\Gamma^{sc}$ has finite index in $\Gamma$.
\par
As recalled in Section~\ref{SS:reduction}, there exists a finite set
$R_0$ of prime ideals of $\Zz_k$ such that $\Gg$ has a model defined
over the ring $\Zz_{k}[1/R_0]$, and such that any two such models are
isomorphic after possibly inverting finitely many more primes.  By
abuse of notation, we will also denote the fixed model by $\G$.  After
possibly increasing $R_0$, we may assume that $\Gamma \subseteq
\G(\Zz_k[1/R_0])$.  In particular, for $\ideal{p}\notin R_0$, we
obtain a well-defined reduction map
$$
\pi_{\ideal{p}}\,:\, \Gamma\ra \G(\Fp_{\ideal{p}}),
$$
and similarly for the simply connected cover $\Gg^{sc}$, and we have
homomorphisms
$$
\varphi_{\ideal{p}}\,:\, \G^{sc}(\Fp_{\ideal{p}})\ra 
\G(\Fp_{\ideal{p}}).
$$
\par
The following deep result, called the ``Strong Approximation
Property'', explains why we need to use $\Gamma^{sc}$: the statement
is false, in general, if $\Gamma^{sc}$ is replaced with $\Gamma$
itself.

\begin{proposition}\label{pr-sc}
  Let $(k, \G, \Gamma)$ be as given, in particular such that $\Gamma$
  is Zariski-dense in $\G$, and let $\G^{sc}$, $\Gamma^{sc}$ be as
  defined above. Let
$$
\Gamma^{sc}_{\ideal{p}}=\pi_{\ideal{p}}(\Gamma^{sc}),
$$
where $\pi_{\ideal{p}}$ is the reduction map defined above for almost
all prime ideals of $\Zz_k$. 
\par
Then there exists a finite subset $R\supset R_0$ of prime ideals,
depending only on $(k,\G, \Gamma,R_0)$, such that for any
$\ideal{p}\notin 
R$, we have
$$
\Gamma^{sc}_{\ideal{p}}=\varphi_{\ideal{p}}(\G^{sc}(\Fp_{\ideal{p}})),
$$
the image of the group of $\Fp_{\ideal{p}}$-rational points of the
simply connected covering of $\G$, and moreover the product maps
$$
\Gamma^{sc}\fleche{\pi_{\ideal{p}}\times \pi_{\ideal{p}'}}
\Gamma^{sc}_{\ideal{p}}\times \Gamma^{sc}_{\ideal{p}'}
$$
for $\ideal{p}\not=\ideal{p}'$, both not in $R$, are surjective.
\end{proposition}

\begin{proof}
  Results of this type, in varying generality, have been proved by
  many people, using a wide variety of techniques; see, e.g., the
  papers of Nori~\cite[Th. 5.1]{nori}, Matthews, Vaserstein and
  Weisfeiler~\cite[Th., p. 515]{mvw}, Weisfeiler~\cite[\S
  9]{weisfeiler}, Hrushovski and Pillay~\cite[Prop. 7.3]{hru-pillay}
  (see also the comments in~\cite[\S 7.5]{platonov-rapinchuk}).
  Precisely, we first apply~\cite[Th. 9.1.1]{weisfeiler} with data
$$
(k,G,\Gamma) = (k,\G^{sc},\varphi^{-1}(\Gamma))
$$
to deduce that $\varphi^{-1}(\Gamma)$ surjects to $\G^{sc}(\Zz_k/I)$
for all integral ideals $I\not=0$ coprime with some finite set of
primes in $\Zz_{k}$. Taking $I=\ideal{p}$ and composing with $\varphi$
and $\varphi_{\ideal{p}}$, respectively, we derive
$\Gamma^{sc}_{\ideal{p}}=\pi_{\ideal{p}}(\Gamma^{sc})$ for
$\ideal{p}\notin R$.
\par
Then, since 
$$
\G^{sc}(\Zz_k/\ideal{p}\ideal{p'})=\G^{sc}(\Zz_k/\ideal{p})
\times \G^{sc}(\Zz_k/\ideal{p'})
$$
if $\ideal{p}\not=\ideal{p}'$ are both prime ideals not in $R$ (by a
straightforward Chinese Remainder Theorem), it also follows that
$\varphi^{-1}(\Gamma)$ surjects onto $\G^{sc}(\Fp_{\ideal{p}})\times
\G^{sc}(\Fp_{\ideal{p}'})$ for $\ideal{p}$ and $\ideal{p}'$ both
outside $R$, and the final conclusion is obtained by applying again
the map $\varphi$.
\end{proof}

\begin{remark}
  Here is an illustration of failure of this result when the group is
  not simply connected: let $Q$ be a nondegenerate indefinite
  quadratic form over $\Zz$ and $\Gg=\SO(Q)$, which is defined over
  $\Zz$. It is a standard fact that the spinor norm of an element in
  the group of integral points $\SO(Q,\Zz)$ is $\pm 1$ (modulo the non
  zero squares of $\Qq^\times$). Thus for any prime $p$ congruent to
  $1$ modulo $4$ (i.e.  for a subset of primes of density $1/2$), the
  image of $\SO(Q,\Zz)$ by reduction modulo $p$ equals the spinorial
  kernel $\Omega(n,\Fp_p)$ which means that for any $p$ congruent to
  $1$ modulo $4$, the morphism of reduction modulo $p$ fails to be
  surjective onto $\SO(Q,\Fp_p)$ and its image has index $2$.
\end{remark}

Since $\Gamma^{sc}$ is of finite index in $\Gamma$, the idea is now to
use random walks on the cosets of $\Gamma^{sc}$ in $\Gamma$ in order
to perform the sieve. Of course, the original walk we wish to consider
is not of this type. One could, as in~\cite[Section 1.1]{jouve} decide
that it is good enough to deal with each fixed coset separately (using
random variables of the type $Y_n=\gamma X_n$ for a fixed $\gamma$ and
a random walk $(X_n)$ on $\Gamma^{sc}$), provided we obtain the
``same'' result, independently of $\gamma$. However, we want to do
better. For this, the idea is also suggested by the (easy) case
of~\cite[Prop. 7.11]{lsieve}, where random walks on $\Sp(4,\Zz)$ were
studied by reducing to auxiliary random walks (namely $Y_n=X_{2n}$ and
$Z_n=X_{2n+1}$) on the two cosets in
$\Sp(4,\Zz)/[\Sp(4,\Zz),\Sp(4,\Zz)]$, when the original walk had the
property that every other step was in the non-identity coset.
\par
We do something similar; we don't know exactly when the finitely many
cosets $\gamma\in \Gamma^{sc}\backslash \Gamma$ are reached, but we
can use probabilistic results to show that every coset is covered
essentially equally often. Note that readers not familiar with the
basic properties of Markov chains (with countable state space) may
wish to assume that $\Gamma^{sc}=\Gamma$ and skip directly to
Corollary~\ref{cor-sieve-semisimple}, reading the latter with this
assumption in mind.\footnote{\ This assumption holds in a number of
  cases, such as $\SL(m)$ or $\Sp(2g)$.}
\par
Let $\mathcal{C}=\Gamma^{sc}\backslash \Gamma$ be the finite set of
cosets; we write $g\equiv \gamma$ to state that $g\in \Gamma$ is in
the coset $\gamma\in \mathcal{C}$ (instead of $g\in \gamma$). Fix
representatives $\tilde{\gamma}$ in $\Gamma$ of all $\gamma\in \mathcal{C}$.
\par
Let now $(\gamma_n)$ be the random walk on the finite set
$\mathcal{C}=\Gamma^{sc}\backslash \Gamma$ induced from the walk $(X_n)$ on
$\Gamma$. In probabilistic terms, $(\gamma_n)$ is a finite Markov
chain with Markov kernel
$$
K(\gamma,\gamma')=\sum_{s\in S,\, \gamma s=\gamma'}{p(s)}.
$$
\par
This Markov chain is \emph{irreducible}, because the possible steps
$S$ of $(X_n)$ have positive probability and generate $\Gamma$, and
\emph{reversible} because the probabilities $p(s)$ satisfy
$p(s)=p(s^{-1})$. The (unique) \emph{stationary distribution}
associated with $(\gamma_n)$ is the uniform distribution on $\mathcal{C}$, i.e.,
we have
$$
\frac{1}{|\mathcal{C}|}\sum_{\gamma\in \mathcal{C}}{K(\gamma,\gamma')}=\frac{1}{|\mathcal{C}|}
$$
for any $\gamma'\in \mathcal{C}$. (For basic facts and terminology, we refer
to~\cite[\S 2]{saloff-coste} or~\cite[Ch. II]{bw}.)
\par
For any $\gamma\in \mathcal{C}$, we define recursively the following sequence of
\emph{random times}
$$
t_{\gamma,j}\,:\, \Omega\ra \{0,1, 2\ldots, \}\cup {\infty},
$$
which indicate for which successive indices the walk falls in $\gamma$:
first
$$
t_{\gamma,0}=\min\{n\geq 0\,\mid\, X_n\equiv \gamma\},
$$
and for $j\geq 0$, we have
$$
t_{\gamma,j+1}=
\begin{cases}
  +\infty&\text{ if } t_{\gamma,j}=+\infty\\
  \min\{n> t_{\gamma,j}\,\mid\, X_n\equiv \gamma\},&\text{ otherwise.}
\end{cases}
$$
\par
We then define auxiliary random walks by
$$
Y_{\gamma,j}=
\begin{cases}
X_{t_{\gamma,j}}\in\gamma&\text{ if } t_{\gamma,j}<+\infty\\
\tilde{\gamma}\in\gamma&\text{ otherwise,}
\end{cases}
$$
where (as we will see immediately) the second case is only present for
definiteness.
\par
These random walks are then quite similar to the original ones, but
(by definition) lie in a single coset of $\Gamma^{sc}$.

\begin{lemma}
With notation as above, we have the following properties:
\par
\emph{(1)} Almost surely, all the $t_{\gamma,j}$ are finite. 
\par
\emph{(2)} For any $\gamma\in \mathcal{C}$, the sequence $(Y_{\gamma,j})_{j\geq
  0}$ is a random walk on the coset $\gamma$, given by the initial
$\gamma$-valued random variable $Y_{\gamma,0}$, and with steps
$$
\beta_{\gamma,j}=Y_{\gamma,j-1}^{-1}Y_{\gamma,j}
$$ 
which are $\Gamma^{sc}$-valued, independent and identically
distributed; their distribution is given by the rule
\begin{equation}\label{eq-steps}
\proba(\beta_{\gamma,j}=g)=\sum_{k\geq 1}{
\sum_{\stacksum{s_1\cdots s_k=g}{
s_1\cdots s_m\notin \Gamma^{sc},\ m<k}}
{
p(s_1)\cdots p(s_k)
}},\quad\quad\text{for } g\in \Gamma^{sc}.
\end{equation}
Moreover, we have
$\proba(\beta_{\gamma,j}=g)=\proba(\beta_{\gamma,j}=g^{-1})$ for any
$g\in \Gamma^{sc}$.
\end{lemma}

\begin{proof}
  Part (1) is a well-known property of finite irreducible Markov
  chains (it is possible to go from any coset to the other), see,
  e.g.,~\cite[Prop. II.8.1]{bw}.
\par
Part (2): from (1), we know that the random walk $(Y_{\gamma,j})$ is
well-defined. Its initial state is $Y_{\gamma,0}$ by
definition. Therefore, it remains to show that the steps
$$
\beta_{\gamma,j}=Y_{\gamma,j-1}^{-1}Y_{\gamma,j}=
X_{t_{\gamma,j-1}}^{-1}X_{t_{\gamma,j}},\quad\text{ for }
j\geq 1
$$
are distributed according to~(\ref{eq-steps}), are independent, and
independent of the initial step $Y_{\gamma,0}$. This is intuitively
natural, and is a fairly standard fact in probability, but we give a
certain amount of details for completeness for those readers who have
not seen this type of arguments before (see, e.g.,~\cite[II,
Th. 4.1]{bw} for similar reasoning).
\par
Of course, $\beta_{\gamma,j}$ is $\Gamma^{sc}$-valued by
construction. We will show that the distribution is the one
claimed. For $g\in \Gamma^{sc}$, we have
\begin{align}
\proba(\beta_{\gamma,j}=g)
&=\proba(X_{t_{\gamma,j}}=X_{t_{\gamma,j-1}}g)\nonumber\\
&=\sum_{k\geq 1}{\proba(t_{\gamma,j}=t_{\gamma,j-1}+k,\text{ and }
X_{t_{\gamma,j-1}+k}=X_{t_{\gamma,j-1}}g)}\nonumber\\
&=\sum_{k\geq 1}{
\sum_{\stacksum{(s_1,\ldots,s_k)\in S^k}{g=s_1\cdots s_k}}
{
\proba(t_{\gamma,j}=t_{\gamma,j-1}+k\text{ and }
\xi_{t_{\gamma,j-1}+m}=s_m,\text{ for } 1\leq m\leq k)
}}\nonumber\\
&=\sum_{k\geq 1}{
\sum_{\stacksum{g=s_1\cdots s_k}{s_1\cdots s_m\notin \Gamma^{sc},\ m<k}}}
{
\proba(\xi_{t_{\gamma,j-1}+m}=s_m,\text{ for } 1\leq m\leq k)
}\label{eq-steps2}
\end{align}
since the condition that $t_{\gamma,j}=t_{\gamma,j-1}+k$ means, by
definition of $t_{\gamma,j}$ and the fact that $X_{t_{\gamma,j-1}}\in
\Gamma^{sc}$, that none of the intermediate elements
$$
\xi_{t_{\gamma,j-1}+1}\cdots
\xi_{t_{\gamma,j-1}+m}=s_1\cdots s_m,
$$
are in $\Gamma^{sc}$ for $1\leq m<k$.
\par
Now we can invoke the \emph{strong Markov property} of the original
random walk~\cite[Th. II.4.1]{bw}, which implies that the random walk
defined by
\begin{equation}\label{eq-markov}
Z_m=X_{t_{\gamma,j-1}+m},\quad\text{ for }m\geq 0
\end{equation}
is itself a random walk on $\Gamma$ with steps which are independent and
distributed like the original steps of $(X_n)$. Note that this would
be obvious if $t_{\gamma,j-1}$ were constant,\footnote{\ In which case
  it is the \emph{Markov property}.} but is false for a general random
time: imagine for instance looking at $Z_m=X_{T+m}$ where the time $T$
is defined to be the least index $n$ such that $\xi_{n+1}=s$ (for some
fixed $s\in S$). The suitable property which holds for the random time
$t_{\gamma,j-1}$ is that it is a \emph{stopping time} for the standard
filtration associated with $(X_n)$, meaning that the events
$$
\{t_{\gamma,j-1}=k\}
$$
for any $k\geq 0$, are measurable for the $\sigma$-field
$\sigma(X_1,\ldots, X_k)$ (which is obvious since determining whether
$t_{\gamma,j-1}=k$ can be done by looking at the first $k$ steps of
the original walk).
\par
>From this, it follows that
\begin{align*}
\proba(\xi_{t_{\gamma,j-1}+m}=s_m,\text{ for } 1\leq m\leq k)&=
\proba(\xi_{m}=s_m,\text{ for } 1\leq m\leq k)
\\
&=p(s_1)\cdots p(s_k),
\end{align*}
and the distribution property~(\ref{eq-steps}) then follows
from~(\ref{eq-steps2}). The symmetry property of the distribution is
obvious.
\par
The independence of the steps $\beta_{\gamma,j}$ is also a consequence
of the strong Markov property and computations very similar to the
previous one, except for notational complications. 
\end{proof}

Note that in these auxiliary walks, the initial distribution depends
on $\gamma$, not the steps of the walk (though, $\mathcal{C}$ being
finite, such a dependency would not affect the remainder of the
argument). A further difference with the original walk $(X_n)$ is the
feature that the steps $\beta_{\gamma,j}$ are supported on the whole
of the discrete group $\Gamma^{sc}$, instead of the original finite
set $S$. This is, however, still a symmetric generating set of
$\Gamma^{sc}$. It turns out that random walks involving infinite
generating sets were also already considered in~\cite[Introduction and
Section 1.2]{jouve}, so we can build on this.
\par
The following general sieve result follows quite simply from the
theory developed in~\cite[\S 7]{lsieve} and the adjustments
in~\cite{jouve}.

\begin{proposition}\label{pr-sieve-sc}
  Let $(k, \G, \Gamma)$ be given as before, and define $\Gamma^{sc}$,
  $\mathcal{C}$, $\Gamma^{sc}_{\ideal{p}}$, $\pi_{\ideal{p}}\,:\,
  \Gamma^{sc}\ra \Gamma^{sc}_{\ideal{p}}$ as above.
\par
Let $(Y_j)$, $j\geq 0$, be a random walk on a fixed coset $\gamma\in
\mathcal{C}$ of $\Gamma^{sc}$, with initial step $Y_0$ and with independent,
identically distributed steps $(\beta_j)$, $j\geq 1$, such that the
support of the law of the $\beta_j$ is a generating set of
$\Gamma^{sc}$, and
$$
\proba(\beta_j=g)=\proba(\beta_j=g^{-1}),\quad\quad
\text{ for } j\geq 1,\ g\in \Gamma^{sc}.
$$
\par
Assume that $(\Gamma^{sc},S)$ is a balanced pair: either
$\proba(\beta_j=1)>0$, or there is no surjection $\Gamma^{sc}\ra
\Zz/2\Zz$.
\par
There exists a finite set $R$ of prime ideals in $\Zz_k$, depending
only on $\Gamma$, and constants $c>0$ and $A\geq 0$, depending only on
$k$, $\Gamma$ and the distribution of the steps $(\beta_j)$ such that
the following holds: for any choice of subsets
$$
\Omega_{\ideal{p}}\subset \tilde{\gamma}\Gamma^{sc}_{\ideal{p}},
$$
invariant under $\G(\Fp_{\ideal{p}})$-conjugation, with
$\Omega_{\ideal{p}}=\emptyset$ if $\ideal{p}\in R$, we have
$$
\proba(\pi_{\ideal{p}}(Y_j)\notin \Omega_{\ideal{p}}\text{ for }
N\ideal{p}\leq L) \leq (1+L^Ae^{-cj}) V^{-1}
$$
for any $L\geq 2$, where
$$
V= \sum_{N\ideal{p}\leq L}
{ \frac{|\Omega_{\ideal{p}}|}{|\Gamma^{sc}_{\ideal{p}}|} }.
$$
\end{proposition}

\begin{proof}
  The main ingredient, beside Proposition~\ref{pr-sc}, is the
  following:
  \par
  [Property $(\tau)$] The group $\Gamma^{sc}$ has Property $(\tau)$ (in the
  sense of Lubotzky) with respect to the family of its congruence
  subgroups, and in particular with respect to the family of subgroups
  of the type
$$
(\ker (\pi_{\ideal{p}}\times \pi_{\ideal{p}'}))
$$
where $\ideal{p}$ and $\ideal{p}'$ run over all prime ideals not in
$R$. This conjecture of Lubotzky and Zimmer was proved by
Clozel~\cite{clozel}.
\par
We then apply the general methods in~\cite[Ch. 3, Ch. 7]{lsieve}
(compare with~\cite[Section 1]{jouve}). More precisely, we first note
that, from Property $(\tau)$, there exists a finite subset (say $S'$)
of $\Gamma^{sc}$ with
$$
\min_{s\in S'}\proba(\beta_j=s)>0
$$
and $\delta>0$ with the following property: for any finite-dimensional
unitary representation
$$
\Gamma^{sc}\fleche{\rho} U(N,\Cc)
$$
that factors through some product of ``prime'' congruence groups,
i.e., $\rho$ is given by
$$
\Gamma^{sc} \ra \Gamma^{sc}/\ker (\pi_{\ideal{p}}\times \pi_{\ideal{p}'})
\fleche{\rho'} U(N,\Cc),
$$
we have 
$$
\min_{s\in S'}\{\|\rho(s)v-v\|\}\geq \delta\|v\|,\quad\quad v\in\Cc^N
$$
provided there is no vector $v\not=0$ which is invariant under $\rho$.
\par
With this, one can follow the proof of~\cite[Prop. 7.2]{lsieve},
or~\cite[Prop. 5]{jouve} to obtain the large sieve bound.
\end{proof}

If we use this technique to control a random walk on $\Gamma$ itself
by splitting into auxiliary walks, this proposition requires one extra
piece of information to be useful: namely, the estimate in terms of
$j$ must be transformed into information in terms of the parameter $n$
of the original walk. Intuitively, since there are $|\mathcal{C}|$
different cosets, and the random walk on $\mathcal{C}$ mixes very
quickly (it is a finite irreducible Markov chain), converging to the
stationary uniform distribution on $\mathcal{C}$ (all cosets being
equally likely), we expect that $X_n$ is, roughly, the $j$-th step of
the auxiliary random walk $Y_{\gamma_n,j}$ for an index $j$ close to
$n/|\mathcal{C}|$.
The following result makes this precise:

\begin{lemma}\label{lm-grand-nombres}
  Let $(k, \G, \Gamma, S, (X_n))$ be given as before, and let
  $(\G^{sc}, \mathcal{C}, (\gamma_n), (t_{\gamma,j})_{\gamma,j},
  (Y_{\gamma,j}))$ be defined as above.
\par
For $n\geq 0$, let $\iota_n$ be the random index such that
$$
X_n=Y_{\gamma_n,\iota_n}.
$$
\par
Then we have
$$
\proba\Bigl(\iota_n<\frac{1}{2}\frac{n+1}{|\mathcal{C}|}\Bigr)\ll \exp(-cn),
$$
for all $n\geq 0$ and some constant $c>0$, depending only on $\mathcal{C}$ and
the distribution of the steps of the Markov chain $\gamma_n$, as does
the implied constant.
\end{lemma}

\begin{proof}
  We can express $\iota_n$ concisely by
$$
\iota_n=|\{i\,\mid\, 0\leq i\leq n,\ X_i\equiv X_n\}|,
\quad\text{ so } \quad
\frac{\iota_n}{n+1}=\frac{1}{n+1} 
\sum_{0\leq i\leq n}{\charfun_{\{\gamma_n\}}(\gamma_i)}.
$$
\par
Now fix $\gamma\in \mathcal{C}$ instead, and consider the deterministic variant
$$ 
\kappa_{\gamma,n}=\frac{1}{n+1}\sum_{0\leq i\leq n}{
  \charfun_{\gamma}(\gamma_i)}.
$$
\par
>From basic properties of Markov chains, we know that
\begin{equation}\label{eq-equi}
\lim_{n\ra +\infty}{\proba(\gamma_n=\gamma)}=\frac{1}{|\mathcal{C}|},
\end{equation}
by equidistribution of the random walk $(\gamma_n)$ on the finite set
$\mathcal{C}$ (in particular, this does not depend on $\gamma$). Noting that
$$
\expect(\charfun_{\gamma}(\gamma_i))=\proba(\gamma_i=\gamma),
$$
we may expect by results like the law of large numbers that
$\kappa_{\gamma,n}$ is usually ``close'' to $\frac{1}{|\mathcal{C}|}$, which
makes it clear intuitively why the probability we look for should be
small.
\par
However, because the random variables $\charfun_{\gamma}(\gamma_i)$ are
not independent, we can not simply apply the standard results about
sums of independent random variables.  But because the convergence to
equidistribution~(\ref{eq-equi}) is exponentially fast, fairly
classical works in probability theory have extended the basic
convergence results (weak and strong law of large numbers, central
limit theorem, large deviations results) to this context. 
\par
We precisely need a \emph{large deviation} result, which in the
simplest (classical) context is the Chernov bound. Here we quote for
concreteness from the explicit result in~\cite{lezaud}, though general
bounds go back to Miller, Gillman, Donsker and Varadhan (and Lezaud's
result has been improved in some contexts by Le\'on and Perron).
\par
From~\cite[Th. 1.1, Remark 3]{lezaud}, we obtain
$$
\proba\Bigl(\kappa_{\gamma,n}<\frac{1}{2}\frac{n+1}{|\mathcal{C}|}\Bigr) \leq
e^{\beta/5}|\mathcal{C}|^{1/2}\exp\Bigl(- \frac{\beta(n+1)}{12\cdot
  (2|\mathcal{C}|)^2}\Bigr)
$$
where $\beta>0$ is the spectral gap of the Markov chain $(\gamma_n)$
(precisely, apply \cite[Remark 3]{lezaud} with the data given by
$$
(G,X_i,\pi)=(\mathcal{C},\gamma_{i+1},\text{the uniform distribution on $\mathcal{C}$})
$$
so that $N_q$ in loc. cit. is bounded by $\sqrt{|\mathcal{C}|}$, and the
function $f$ is given by
$$
f(g)=\frac{1}{|\mathcal{C}|}-\charfun_{\gamma}(g),\quad \text{ for }
g\in \mathcal{C},
$$
while the constant denoted $\gamma$ in loc. cit. is
$(2|\mathcal{C}|)^{-1}$). Note the upper bound we derived does not depend on
$\gamma\in \mathcal{C}$.
\par 
Finally, to come back to the actual index $\iota_n$, we simply write
\begin{align*}
  \proba\Bigl(\iota_n<\frac{1}{2}\frac{n+1}{|\mathcal{C}|}\Bigr)&=
  \sum_{\gamma\in \mathcal{C}}{
    \proba\Bigl(\kappa_{\gamma,n}<\frac{1}{2}\frac{n+1}{|\mathcal{C}|} \text{
      and } \gamma_n=\gamma\Bigr)
  }\\
  &\leq \sum_{\gamma\in \mathcal{C}}{
    \proba\Bigl(\kappa_{\gamma,n}<\frac{1}{2}\frac{n+1}{|\mathcal{C}|}\Bigr)
  }\\
  &\leq e^{\beta/5}|\mathcal{C}|^{3/2}\exp\Bigl(- \frac{\beta(n+1)}{12\cdot
    (2|\mathcal{C}|)^2}\Bigr).
\end{align*}
\par
Cleaning up the constants, this clearly implies the result as stated
(and is in fact much more precise).
\end{proof}

This lemma means that, except for exceptions occurring with
exponentially decaying probability, the sieve statement for the
auxiliary walks leads to a sieve statement for the original one, where
the dependency on the lenght behaves as expected.

\begin{corollary}\label{cor-sieve-semisimple}
  Let $k$, $\Zz_k$, $\G$, $\Gamma$, $\Gamma^{sc}$, $\pi_{\ideal{p}}$,
  $\Gamma^{sc}_{\ideal{p}}$, $\mathcal{C}$ be as above, in particular $\Gamma$
  is Zariski-dense in $\G$.
\par
Let $(X_n)$, $n\geq 0$, be a random walk on $\Gamma$ starting at the
origin with independent, identically distributed steps $\xi_n$, $n\geq
1$, supported on a finite symmetric generating set $S$ of $\Gamma$
with
$$
\proba(\xi_n=s)=\proba(\xi_n=s^{-1})>0,\quad\quad \text{ for } n\geq
1,\ s\in S,
$$
and such that $(\Gamma,S)$ is balanced in the sense described at the
beginning of Section~\ref{sec-sieve}.
\par
There exists a finite set $R$ of prime ideals in $\Zz_k$, depending
only on $\Gamma$, and constants $c>0$ and $A\geq 0$, $B\geq 0$,
depending only on $k$, $\Gamma$ and the distribution of the steps
$(\xi_n)$ such that the following holds: for any choice of subsets
$$
\Omega_{\ideal{p}}\subset \Gamma_{\ideal{p}}=\pi_{\ideal{p}}(\Gamma),
$$
invariant under $\Gamma_{\ideal{p}}$-conjugation, with
$\Omega_{\ideal{p}}=\emptyset$ if $\ideal{p}\in R$, we have
\begin{equation}\label{eq-sieve-semisimple}
\proba(\pi_{\ideal{p}}(X_n)\notin \Omega_{\ideal{p}}\text{ for }
N\ideal{p}\leq L) \leq 
Be^{-cn}+ n(1+L^Ae^{-cn}) \sum_{\gamma\in \mathcal{C}}\frac{1}{V_{\gamma}}
\end{equation}
for any $L\geq 2$, where
$$
V_{\gamma}= 
\sum_{N\ideal{p}\leq L} { 
\frac{|\Omega_{\ideal{p}}\cap \gamma \Gamma^{sc}_{\ideal{p}}|} 
{|\Gamma^{sc}_{\ideal{p}}|}}.
$$
\end{corollary}

\begin{proof}
  With the notation for the auxiliary walks $(Y_{\gamma,j})$
  previously introduced, and writing
$$
\Omega_{\gamma,\ideal{p}}=\Omega_{\ideal{p}}\cap \gamma
\Gamma^{sc}_{\ideal{p}},
$$
the event considered is the disjoint union, over $\gamma\in \mathcal{C}$ and
$j\geq 0$, of the events
$$
S_{\gamma,j}=\{\iota_n=j\text{ and }
\pi_{\ideal{p}}(Y_{\gamma,j})\notin \Omega_{\gamma,\ideal{p}}\text{ for }
N\ideal{p}\leq L\}.
$$
\par
We have $\iota_n\leq n$ and hence $S_{\gamma,j}=\emptyset$ for all
$j>n$. Moreover, by Lemma~\ref{lm-grand-nombres}, the probability of
the union of all $S_{\gamma,j}$ with $j<\demi \tfrac{n}{|\mathcal{C}|}$ is at
most
$$
\proba\Bigl(\iota_n<\frac{1}{2}\frac{n}{|\mathcal{C}|}\Bigr)\ll \exp(-c_1n),
$$
for some constant $c_1>0$.  For others values of $j$, we have
\begin{align*}
  \proba(S_{\gamma,j})&\leq \proba(\pi_{\ideal{p}}(Y_{\gamma,j})\notin
  \Omega_{\gamma,\ideal{p}}\text{ for } N\ideal{p}\leq L) \leq
  (1+L^A\exp(-c_2j))\Bigl( \sum_{N\ideal{p}\leq L} {
    \frac{|\Omega_{\gamma,\ideal{p}}|}{|\Gamma^{sc}_{\ideal{p}}|} }
  \Bigr)^{-1}
\intertext{by Proposition~\ref{pr-sieve-sc}, and this is}
  &\leq \Bigl(1+L^A\exp\Bigl(-\frac{c_2n}{2|\mathcal{C}|}\Bigr)\Bigr)\Bigl(
  \sum_{N\ideal{p}\leq L} {
    \frac{|\Omega_{\gamma,\ideal{p}}|}{|\Gamma^{sc}_{\ideal{p}}|} }
  \Bigr)^{-1},
\end{align*}
for some constant $c_2>0$.
\par
Summing over the values of $j$ and $\gamma$, we obtain the desired
statement, with the constant $c$ given by $\min(c_1,c_2/(2|\mathcal{C}|))$. 
\end{proof}

\begin{remark}
  This result is slightly weaker than the sieve bound for the simply
  connected case, but it is very close in applications. The
  intervention of the cosets $\gamma$ in the sieve bound can not be
  dispensed with in general (i.e., Proposition~\ref{pr-sieve-sc} fails
  if $\Gamma^{sc}$ is replaced with $\Gamma$): suppose, say, that
  $|\mathcal{C}|=2$ with $\Gamma^{sc}_{\ideal{p}}$ also of index $2$ in
  $\Gamma_{\ideal{p}}$; then, if $\Omega_{\ideal{p}}$ is the
  non-trivial coset of $\Gamma^{sc}_{\ideal{p}}$, we have
$$
\proba(\pi_{\ideal{p}}(X_n)\notin \Omega_{\ideal{p}} \text{ for
}N\ideal{p}\leq L)\geq \proba(X_n\in \Gamma^{sc})
$$
which typically converges to $\demi$ as $n\ra +\infty$, while an
hypothetical estimate like
$$
(1+L^A\exp(-cn))\Bigl(\sum_{N\ideal{p}\leq L}{
\frac{|\Omega_{\ideal{p}}|}
{|G_{\ideal{p}}|}}\Bigr)^{-1}=
2(1+L^A\exp(-cn))\pi(L)^{-1}
$$
with $c>0$ for this probability would show it to go to zero
exponentially fast as $n\ra +\infty$, after selecting
$L=\exp(cn/A)$. (On the other hand, in
Corollary~\ref{cor-sieve-semisimple}, the term involving the
non-trivial coset $\Gamma-\Gamma^{sc}$ is the inverse of
$$
\sum_{N\ideal{p}\leq L}{ \frac{|\Omega_{\ideal{p}}\cap
    (\Gamma_{\ideal{p}}-\Gamma^{sc}_{\ideal{p}})|} {|\Gamma^{sc}_{\ideal{p}}|}}=0,
$$
so the proposition does work in this context.)
\end{remark}

\begin{remark}\label{rm-reductive-problem}
  We are considering semisimple groups here because sieving in all
  arithmetic subgroups of reductive groups is problematic for
  well-known reasons: if there is a non-trivial central torus
  $\T\subset Z(\G)$, sieving questions might involve unknown issues
  like the existence of infinitely many Mersenne primes. Moreover,
  although it is tempting to try to apply once more the strategy 
described in the discussion following Proposition~\ref{pr-sc}, 
 the subgroup $Z(\G)$ does not necessarily have finite index in $\G$
  (say if $\G=\GL(n)$ and $k$ is a real quadratic field, so that
  there are infinitely many units), and usually the random walk will
  not come back infinitely often to each coset.
\end{remark}

\begin{remark}
  It is very likely that what we have done in this section is valid
  when $\Gamma$ is simply a finitely generated Zariski-dense subgroup of
  $\G(\Zz_k)$, but not necessarily of finite index (due to recent
  breakthroughs by Helfgott~\cite{helfgott},
  Bourgain--Gamburd~\cite{b-g}, Breuillard--Green--Tao~\cite{bgt} and
  Pyber--Szab\'o~\cite{psz} playing a key role in the proof by Salehi-Golsefidy--Varj\'u~(\cite{salehi-g-varju}) that Zariski-dense subgroups of
  arithmetic groups have Property $(\tau)$ with respect to congruence
  subgroups.)
\end{remark}

\section{Splittings fields of elements of reductive groups}
\label{sec-final}

In this section, we will prove our main theorem which generalizes
Theorem~\ref{th-main-easy}. Let $\G$ be a connected linear algebraic group
defined over a number field $k$.   Let $\Gamma\subseteq \G(k)$ be an
arithmetic subgroup of $\G$ and let $S$ be a finite symmetric set of
generators, such that $(\Gamma,S)$ is balanced (in the sense of the
beginning of Section~\ref{sec-sieve}).  Assume that $\Gamma$ is
Zariski-dense in $\G$.
 
Let $\rho\colon \G\ra \GL(m)$ be a faithful representation of $\G$
defined over $k$. For each element $g\in \G(k)$, the field $k_g$ is
defined as the splitting over $k$ of the polynomial
$\det(T-\rho(g))\in k[T]$.  From Lemma~\ref{L:kg facts}(i), we
know that $k_g$ does not depend on the choice of $\rho$.
 
As in \S\ref{sec-sieve}, let $(\xi_n)$ be a sequence of independent,
identically distributed, random variables taking values in $S$ such
that $\proba(\xi_n=s)=\proba(\xi_n=s^{-1})>0$ for all $s\in S$.  The
sequence $(X_n)$ defined recursively by
$$
X_0=1\in \Gamma,\quad\quad X_{n+1}=X_n\xi_{n+1},
$$
gives a random walk on $\Gamma$.
  
For a reductive group $\G$, we defined in \S\ref{galoistori} an
extension $k_\G/k$ and groups $W(\G)$ and $\Pi(\G)$.  

When $\G$ is not reductive, we set $k_\G:= k_{\G/R_u(\G)}$ and make
the ad hoc definitions $W(\G):=W(\G/R_u(\G))$ and
$\Pi(\G):=\Pi(\G/R_u(\G))$, where $R_u(\G)$ is the unipotent radical
of $\G$.

\begin{theorem}\label{th-main2}  
Fix notation and assumptions as above.
\begin{romanenum}
\item
We have
\[
\lim_{n\to \infty} \proba\big( \Gal(k_{X_n}/k)\cong \Pi(\G) \big) = 1.
\]
\item
If $\G$ is semisimple, then there exists a constant $c>1$ such that
\[
\proba\big( \Gal(k_{X_n}/k)\cong \Pi(\G) \big)= 1 + O(c^{-n})
\]
for all $n\geq 1$. 
\item There exists a constant $c>1$ such that
\[
\proba\big( \Gal(k_{\G}k_{X_n}/k_\G)\cong W(\G) \big) = 1 + O(c^{-n})
\]
for all $n\geq 1$.  
\end{romanenum}
\par
The constants $c$ and the implicit constants depend only on the group
$\G$, the generating set $S$, and the distribution of the $\xi_n$.
\end{theorem}

Theorem~\ref{th-main-easy}(i) is a consequence of
Lemma~\ref{lm-tg-dg}(i).  We obtain the remaining parts of
Theorem~\ref{th-main-easy} from Theorem~\ref{th-main2} by taking
$p(s)=|S|^{-1}$ for all $s\in S$, which then implies by definition
that
$$
\proba(X_n\in A)=\frac{1}{|S|^n}|\{ w=(s_1,\ldots, s_n)\in S^n\,\mid\,
s_1\cdots s_n\in A\}|
$$
for any set $A\subset \Gamma$.    
\par
Our first lemma is a version, in our context, of a
``non-concentration'' estimate on subvarieties.

\begin{lemma}\label{L:closed}  
  Keep the set-up as above and let $Y\subsetneq \G$ be a closed
  subvariety that is stable under conjugation by $\G$.
\begin{romanenum}
\item
We have $\lim_{n\to\infty} \proba\big(X_n \in Y(k) \big) =0$.
\item
If $\G$ is semisimple, then there exists a constant $c>1$ such that
\[
\proba\big(X_n \in Y(k) \big) = O(c^{-n}).
\]
The constant $c$ and the implicit constant depend only on the group $\G$,  the generating set $S$, the distribution of the $\xi_n$, and
$Y$.
\end{romanenum}
\end{lemma}

\begin{proof}
  We start with the proof of (ii), which is more precise.  Choose a
  model $\mathcal{G}$ over $\Zz_k[R^{-1}]$ of $\G$ where $R$ is a
  finite set of maximal ideals of $\Zz_k$.  Since $\Gamma$ is finitely
  generated, we can choose $R$ so that $\Gamma\subseteq
  \mathcal{G}(\Zz_k[R^{-1}])$.  For $\p\notin R$, let $\pi_{\p}\colon
  \Gamma \to \mathcal{G}(\Fp_{\p})$ be the reduction modulo $\p$ map.
  Let $\mathcal{Y}$ be the Zariski closure of $Y$ in $\mathcal{G}$,
  and for each $\p \notin R$, define the set
\[
\Omega_{\p} = \mathcal{G}(\Fp_{\p}) - \mathcal{Y}(\Fp_{\p}).
\]
Our assumption that $Y$ is stable under conjugation by $\G$ implies
that $\Omega_{\p}$ is stable under conjugation by
$\mathcal{G}(\Fp_{\p})$.  If $X_n\in \Gamma$ belongs to $Y(k)$, then
$\pi_{\p}(X_n) \notin \Omega_{\p}$ for all $\p\not \in R$.  Thus for
any $L\geq 2$, we have
\[
\proba\big(X_n \in Y(k) \big) \leq
\proba\bigl(\pi_{\ideal{p}}(X_n)\notin \Omega_{\ideal{p}}\text{ for
  all } \ideal{p}\notin R\text{ with } N\ideal{p}\leq L\bigr).
\]
We are now in a position to apply~(\ref{eq-sieve-semisimple}) to
derive the upper bound
\[
\proba\big(X_n \in Y(k) \big) \leq Be^{-cn}+ n(1+L^Ae^{-cn})
\sum_{\gamma\in \mathcal{C}}\frac{1}{V_{\gamma}},
\]
where $A\geq 0$, $B\geq 0$ and $c>0$ are constants depending only on
$k$, $\mathcal{C}=\Gamma^{sc}\backslash\Gamma$, and
\[
V_{\gamma}:=\sum_{\stacksum{N\ideal{p}\leq L}{\ideal{p}\notin R}}{
  \frac{|\Omega_{\ideal{p}}\cap \gamma \Gamma^{sc}_{\ideal{p}}|}
  {|\Gamma^{sc}_{\ideal{p}}|} } =\sum_{\stacksum{N\ideal{p}\leq
    L}{\ideal{p}\notin R}} \Big( 1- \frac{|\mathcal{Y}(\Fp_{\p})\cap
  \gamma \Gamma^{sc}_{\ideal{p}}|} {|\Gamma^{sc}_{\ideal{p}}|} \Big).
\]
\par
Since $Y$ has smaller dimension that $\G$, the Lang-Weil bounds (see,
e.g.,~\cite[p. 628]{katz}) imply that
$$
|\mathcal{Y}(\Fp_{\p})|/|\mathcal{G}(\Fp_{\p})| = O(1/N(\ideal{p})).
$$
\par
This, together with $[\mathcal{G}(\Fp_{\p}) : \Gamma^{sc}_{\p}] \ll
1$, gives
\[
V_\gamma \gg \sum_{\stacksum{N\ideal{p}\leq L}{\ideal{p}\notin R}}
\Big( 1 + O\big(N(\p)^{-1}\big) \Big) \gg L/\log L,
\]
where the last inequality holds for all $L$ sufficiently large.  Thus
$$
\proba\big(X_n \in Y(k) \big) \ll e^{-cn}+
\frac{n(1+L^Ae^{-cn})}{L}(\log L)
$$
for $L$ sufficiently large.  Setting $L=\exp(nc_1/A)$ with $c_1<c$ a
positive number, gives
$$
\proba\big(X_n \in Y(k) \big) \ll e^{-cn} +
n^2 e^{-c_1n/A} + n^2 e^{-(c-c_1)n} e^{-c_1 n/A}
$$ 
for $n$ large enough.  So there is a $c_2>0$ such that $\proba\big(X_n
\in Y(k) \big) \ll e^{-c_2n}=(e^{-c_2})^n$ for all $n\geq 1$.
\par
Now we come to (i). Let $\G_1$ be the derived group of $\G$, $\T$ the
connected component of the center of $\G$, so that we have a
surjective product map
$$
\G_1\times \T\ra \G
$$
with finite fibers. We are going to reduce the question to one on
$\G_1\times \T$. For this, we observe that there exists a fixed number
field $k_S/k$ and elements $(x_s, z_s)\in (\G_1\times \T)(k_S)$ for
all $s\in S$ such that
$$
s=x_sz_s
$$
for all $s$. We can then write $X_n=Z_n\tilde{X}_n$ with random
variables 
$$
\tilde{X}_n=x_{\xi_1}\cdots x_{\xi_n},\quad\quad
Z_n=z_{\xi_1}\cdots z_{\xi_n}
$$
taking values in $(\G_1\times \T)(k_S)$, which form random walks on
(subgroups of) $\G_1(k_S)$ and $\T(k_S)$.
\par
Now let $\tilde{Y}$ be the inverse image of $Y$ in $\G_1\times \T$. If
the projection of $\tilde{Y}$ on $\G_1$ is not dense, say it is
contained in a proper (conjugacy-invariant) subvariety $Y_1\subset
\G_1$, the condition $X_n\in Y(k)$ implies $\tilde{X}_n\in Y_1(k_S)$,
which occurs with probability tending to $0$ by applying (ii) (with
$k$ replaced by $k_S$). Otherwise, the projection of $\tilde{Y}$ on
the torus $\T$ must be contained in a proper subvariety, say $Y_2$,
and the condition $X_n\in Y(k)$ implies $Z_n\in Y_2(k_S)$. In fact,
$Z_n$ lies in the finite rank abelian group generated by the $z_s$ in
$\T$ (in fact, $\T(k_S)$ is itself of finite rank, by the generalized
Dirichlet unit theorem, see, e.g.,~\cite[Cor. 1,
p. 209]{platonov-rapinchuk}), and we are therefore reduced to a
question that can be handled by more classical sieve methods, for
instance by the large sieve on $\Zz^n$, as described in~\cite[\S
4.2]{lsieve}. Using reductions modulo primes (of $k_S$) and the
Lang-Weil estimate for $Y_2$ to estimate the number of permitted
residue classes for $Y_2(k_S)$, we obtain the qualitative estimate
(i).
\end{proof}

\begin{remark}
We used the sieve result of the previous section, but one could also
deal with this by selecting a single well-chosen prime ideal. We also
see clearly that (i) could be replaced, with some work, by an estimate
of quantitative decay, which would however only be of the type
$n^{-c}$ for some fixed $c>0$. 
\end{remark}

The following proposition, which is given for semisimple groups, will
be key in the proof of Theorem~\ref{th-main2}.  The proof follows the
same basic principle as earlier works using the large sieve to study
probabilistic Galois theory: the sieve implies that Frobenius elements
in $\Gal(k_{X_n}/k)$ can be found (with very high probability) that
map to any given conjugacy class of $W(\G)$ under the injective
homomorphism of Section~\ref{SS:finite fields}, and we can then use
the well-known lemma of Jordan according to which, in a finite group,
no proper subgroup contains elements of all conjugacy classes.

\begin{proposition}\label{pr-main}  
  Fix notation and assumptions as above, and assume that $\G$ is
  semisimple.  Let $K\subseteq \kbar$ be a finite extension of $k_\G$.
  Then there exists a constant $c>1$ such that
\[
\proba\big( \Gal(K k_{X_n}/K)\cong W(\G) \big) = 1 +O(c^{-n}).
\]
\par
The constant $c$ and the implicit constant depend only on the group
$\G$, the generating set $S$, the distribution of the $\xi_n$, and the
field $K$.
\end{proposition}
  
 \begin{proof}
   Fix a maximal torus $\T_0$ of $\G$.  By Lemma~\ref{lm-tg-dg}, the
   group $\Gal(k_\G k_{X_n}/k_\G)$ is isomorphic to a subquotient of
   $W(\G)$.  So without loss of generality, we may extend $K$ so that
   $\T_{0,K}$ is split.  Choose a semisimple group scheme
   $\mathcal{G}$ over $\Zz_k[R^{-1}]$ whose generic fiber is $\G$
   where $R$ is a finite set of maximal ideals of $\Zz_k$.  Let
   $\calT_0$ be the Zariski closure of $\T_0$ in $\calG$.  By taking
   $R$ large enough, we may assume that $\calT_0$ is a maximal torus
   of $\calG$ and $\Gamma\subseteq \mathcal{G}(\Zz_k[R^{-1}])$.  For
   $\p\notin R$, let $\pi_{\p}\colon \Gamma \to \mathcal{G}(\Fp_{\p})$
   be the homomorphism of reduction modulo $\p$.
\par
Let $\mathcal{P}$ be the set of maximal ideals $\p\notin R$ of $\Zz_k$
that split completely in $K$.  For $\p\in \mathcal{P}$, the tori
$\calT_{0,k_\p}$ and $\calT_{0,\FF_\p}$ are split.  The set $\calP$
has positive natural density, by the Chebotarev density theorem (see,
e.g.,~\cite[p. 143]{analyticnumbertheory}.)  For each $\p\in
\mathcal{P}$, fix an embedding $\kbar\hookrightarrow \kbar_\p$ which
is the identity map on $k$.  Let $W(\G)^\sharp \leftrightarrow
W(\mathcal{G}_{\Fp_{\p}})^\sharp$ be the bijection (\ref{E:conjugacy
  bijection 2}); we will use it as an identification.
\par
For $\p \in \mathcal{P}$, we can define a map
\begin{align*}
  \theta_{\p} \colon \mathcal{G}(\Fp_{\p})_{sr} & \to
  W(\mathcal{G}_{\Fp_{\p}})^\sharp =W(\G)^\sharp
\end{align*}
as in \S\ref{SS:finite fields}.  Fix a conjugacy class $C\in
W(\G)^\sharp$.  For $\p\in \calP$, define the set
\[
\Omega_{\p} :=  \big\{ g \in
\mathcal{G}(\Fp_{\p})_{sr} :  \theta_{\p}(g)=C
\}.
\]
\par
Let $\G(k)_{sr}$ be the set of $g\in \G(k)$ that are semisimple and
regular in $\G$.  For $X_n\in \G(k)_{sr}$, we have defined a
representation $\varphi_{X_n}\colon \Gal(\kbar/k) \to \Pi(\G)$ that is
uniquely defined up to conjugation by an element of $W(\G)$.  Fix a
prime $\p\in \calP$ such that $\pi_{\p}(X_n) \in \Omega_{\p}$.  By
Proposition~\ref{P:Frobenius comparison 1}, $\varphi_{X_n}$ is
unramified at $\p$ and the conjugacy class of the Frobenius
automorphism $\Frob_{\p}$ at $\p$ is $C$, i.e.,
$\varphi_{X_n}(\Frob_{\p})=C$.  
\par
Since $\p$ splits completely in $K$, we deduce  that
$$
\varphi_{X_n}(\Gal(\kbar/K))\cap C \neq \emptyset,
$$
and therefore, we have an upper bound
\begin{align} \label{E:pC stuff} \proba\big(X_n\in \G(k)_{sr}
  \text{ and} &\text{ $\varphi_{X_n}(\Gal(\kbar/K))\cap C =
    \emptyset$}\big)\\
 & \leq \proba\bigl( \pi_{\ideal{p}}(X_n) \notin
  \Omega_{\ideal{p}}\text{ for all } \ideal{p} \in \mathcal{P} \text{
    with } N\ideal{p}\leq L\bigr),
\notag
\end{align}
where the last probability is amenable to sieve.
Specifically, applying~(\ref{eq-sieve-semisimple}), we derive the
upper bound
\[
\proba\bigl( \pi_{\ideal{p}}(X_n) \notin
\Omega_{\ideal{p}}\text{ for all } \ideal{p} \in \mathcal{P} \text{
  with } N\ideal{p}\leq L\bigr)
\leq Be^{-cn}+ n(1+L^Ae^{-cn}) \sum_{\gamma\in \mathcal{C}}\frac{1}{V_{\gamma}},
\]
where $A\geq 0$, $B\geq 0$ and $c>0$ are constants depending only on
$k$, $\mathcal{C}=\Gamma^{sc}\backslash\Gamma$, and
\[
V_{\gamma}:=\sum_{\stacksum{N\ideal{p}\leq L}{\ideal{p}\in \mathcal{P}}}{
\frac{|\Omega_{\ideal{p}}\cap \gamma \Gamma^{sc}_{\ideal{p}}|}
{|\Gamma^{sc}_{\ideal{p}}|}
}.
\]
By Proposition~\ref{P:Carter approximation2}, we have
\[
\frac{|\Omega_{\ideal{p}}\cap \gamma \Gamma^{sc}_{\ideal{p}}|}
{|\Gamma^{sc}_{\ideal{p}}|} = \frac{|C|}{|W(\G)|} + O(N(\p)^{-1})
\]
for all $\gamma\in \mathcal{C}$ and $\p \in \mathcal{P}$, where the
implicit constant does not depend on $\p$.  This implies that
\[
V_\gamma \gg \sum_{\stacksum{N\ideal{p}\leq L}{\ideal{p}\in
    \mathcal{P}}}
 \Big(  \frac{|C|}{|W(\G)|} + O\big(N(\p)^{-1}\big) \Big) \gg L/\log L
\]
where the last inequality holds for all $L$ sufficiently large (since
$\mathcal{P}$ has positive density).  Therefore,
$$
\proba\bigl( \pi_{\ideal{p}}(X_n) \notin
\Omega_{\ideal{p}}\text{ for all } \ideal{p} \in \mathcal{P} \text{
  with } N\ideal{p}\leq L\bigr)
\ll e^{-cn}+ \frac{n(1+L^Ae^{-cn})}{L}(\log L).
$$
for all $L$ sufficiently large.  As at the end of the proof of
Lemma~\ref{L:closed}, we obtain
$$
\proba\bigl( \pi_{\ideal{p}}(X_n) \notin \Omega_{\ideal{p}}\text{ for
  all } \ideal{p} \in \mathcal{P} \text{ with } N\ideal{p}\leq L\bigr)
\ll c^{-n}
$$ 
for some $c>1$.

So from (\ref{E:pC stuff}), we find that
\[
\proba\big(X_n\in \G(k)_{sr} \text{ and
  $\varphi_{X_n}(\Gal(\kbar/K))\cap C = \emptyset$}\big) \ll c^{-n}
\]
for some constant $c>1$, which we may assume holds for all $C\in
W(\G)^\sharp$.  By Jordan's lemma, no proper subgroup of $W(\G)$
intersects every conjugacy class of $W(\G)$.  Therefore,
\begin{multline*}
  \proba\big(X_n \in \G(k)_{sr} \text{ and }
  \varphi_{X_n}(\Gal(\kbar/K))\neq W(\G)\big) \leq \\
\sum_{C\in
    W(\G)^\sharp}\proba (X_n \in \G(k)_{sr} \text{ and }
  \varphi_{X_n}(\Gal(\kbar/K))\cap C=\emptyset)\ll c^{-n}.
\end{multline*}
\par
Now let $Y$ be the subvariety of $\G$ from Lemma~\ref{lm-tg-dg}(iii).
By Lemma~\ref{L:closed}, we have 
$$
\proba(X_n\in Y(k))\ll c^{-n},
$$
(after possibly increasing $c>1$).  If $X_n \notin Y(k)$ and
$\varphi_{X_n}(\Gal(\kbar/K))= W(\G)$, then $X_n$ is regular and
semisimple in $\G$ and $\Gal(Kk_{X_n}/K) \cong W(\G)$.  Therefore,
\[
\proba\big( \Gal(Kk_{X_n}/K) \not \cong W(\G) \big)  \ll c^{-n}
\]
for some constant $c>1
$.
\end{proof}

\subsection{Proof of Theorem~\ref{th-main2}}

We first consider the case where $\G$ is reductive.  Let $R(\G)$ be
the \emph{radical} of $\G$.  Since $\G$ is reductive, $R(\G)$ is the
connected component of the center of $\G$.  The quotient
$\G':=\G/R(\G)$ is defined over $k$ and is semisimple. Let $\pi\colon
\G \to \G'$ be the quotient homomorphism.

We now consider $\Gamma'\subseteq \G'(k)$, the image of $\Gamma$ under
$\pi$, and the generating set $S'$ which is the image of $S$.  The
pair $(\Gamma',S')$ is still balanced. The group $\Gamma'$ is Zariski
dense in $\G'$ since $\Gamma$ is Zariski dense in $\G$ and $\pi$ is
surjective; again because $\pi$ is surjective, and $\Gamma$ is
arithmetic, we find that $\Gamma'$ is an arithmetic subgroup of $\G'$.

To the random walk $(X_n)$ on $\Gamma$, we can associate the random
walk $(X_n')$ on $\Gamma'$ where $X_n'=\pi(X_n)$. It is a
left-invariant random walk defined by the sequence of steps $(\xi_n')$
where $\xi_n'=\pi(\xi_n)$. Each $\xi_n'$ takes values in the symmetric
generating system $S'$ of $\Gamma'$ and has distribution
$$
\proba(\xi_n'=s')=
\sum_{s\in S,\, \pi(s)=s'}{p(s)}
$$
for $s' \in S'$.  We have
$\proba(\xi_n'=s')=\proba(\xi_n'=(s')^{-1})>0$ for all $s'\in S'$, and
the random variables $(\xi_n')$ are independent and identically
distributed.  

\begin{lemma} \label{L:kX inclusion}
We have $k_{X_n'} \subseteq k_{X_n}$.
\end{lemma}
\begin{proof}
  More generally, we claim that $k_{\pi(g)}\subseteq k_g$ for all
  $g\in \G(k)$.  Without loss of generality, we may assume that $g$,
  and hence $\pi(g)$, is semisimple.  Let $\T$ be a maximal torus of
  $\G$ containing $g$.  The torus $\T':=\T/R(\G)$ is then a maximal
  torus of $\G'$ which contains $\pi(g)$.  The homomorphism $X(\T')\to
  X(\T),\, \chi' \mapsto \chi'\circ \pi$, gives an inclusion
\[
\{\chi'(\pi(g)) : \chi' \in X(\T')\} \subseteq \{\chi(g):\chi \in X(\T)\}.
\]
By Lemma~\ref{lm-tg-dg}(ii), we deduce that $k_{\pi(g)}\subseteq k_g$.
\end{proof}

Fix a finite extension $K$ of $k$ that contains $k_{\G}$ and
$k_{\G'}$.  Suppose that $\Gal(K k_{X_n'}/K)\cong W(\G')$.  By
Lemma~\ref{L:kX inclusion}, we have $|\Gal(K k_{X_n}/K)| \geq
|W(\G')|$.  Since $\G$ and $\G'$ have isomorphic Weyl groups, we have
$|\Gal(K k_{X_n}/K)| \geq |W(\G)|$.  By Lemma~\ref{lm-tg-dg}(i),
$\Gal(K k_{X_n}/K)$ is isomorphic to a subquotient of $W(\G)$, so by
cardinality considerations we find that $\Gal(K k_{X_n}/K)\cong
W(\G)$.

Therefore, 
\[
\proba\big( \Gal(K k_{X_n'}/K)\cong W(\G') \big) \leq \proba\big(
\Gal(K k_{X_n}/K)\cong W(\G) \big)
\]
\par
Since $\G'$ is semisimple, Proposition~\ref{pr-main} implies that
$$
\proba\big( \Gal(K k_{X_n'}/K)\cong W(\G') \big)=1+O(c^{-n})
$$
for some constant $c>1$, and therefore
$$
\proba\big( \Gal(Kk_{X_n}/K)\cong W(\G) \big)=1+O(c^{-n}).
$$
\par
Since $\Gal(k_{\G} k_{X_n}/k_\G)$ is always isomorphic to a
subquotient of $W(\G)$ by Lemma~\ref{lm-tg-dg}(i), we deduce that
$$
\proba\big( \Gal(k_\G k_{X_n}/k_\G)\cong W(\G) \big)=1+O(c^{-n}),
$$
and this completes the proof of (iii) in the reductive case.

Let $Y$ be the subvariety of $\G$ from Lemma~\ref{lm-tg-dg}(iii).  Fix
$g\in \G(k)-Y(k)$ such that 
$$
\Gal(k_\G k_g/k_\G)\cong W(\G).
$$
\par
We claim that $\Gal(k_g/k)\cong \Pi(\G)$. Since $g\notin Y(k)$, $g$ is
contained in a unique maximal torus $\T$ of $\G$ and $k_g=k_{\T}$.  It
thus suffices to show that $\Gal(k_\T/k)\cong \Pi(\G)$.  The
homomorphism 
$$
\Gal(k_\T/k) \xhookrightarrow{\varphi_\T} \Pi(\G)\to
\Pi(\G)/W(\G)
$$ 
is surjective and $\varphi_\T(\Gal(k_T/k_\G))\subseteq W(\G)$, so it
suffices to show that $\Gal(k_T/k_\G)\cong W(\G)$.  But since
$k_\T\supseteq k_\G$, we have $\Gal(k_\T/k_\G)=\Gal(k_\G
k_g/k_\G)\cong W(\G)$ as desired.  Therefore
\[
\proba(\Gal(k_{X_n}/k)\not\cong \Pi(\G)) \leq \proba(\Gal(k_\G
k_{X_n}/k_\G)\not\cong W(\G)) + \proba(X_n \in Y(k)).
\]
\par
By part (iii), which we have already proved, we have
\[
\proba(\Gal(k_{X_n}/k)\not\cong \Pi(\G)) \ll c^{-n} + \proba(X_n \in Y(k))
\]
for some constant $c>1$.  Part (i) and (ii) in the reductive case then
follow immediately from Lemma~\ref{L:closed}.
\par
Finally, we consider the case where $\G$ is not reductive.  The
quotient $\G':=\G/R_u(\G)$ is defined over $k$ and is reductive. Let
$\pi\colon \G \to \G'$ be the quotient homomorphism.  As above, we can
consider the arithmetic subgroup $\Gamma':=\pi(\Gamma)$ of $\G'(k)$
and the related random walk $(X_n')$ on $\Gamma'$ where
$X_n'=\pi(X_n)$.  By Lemma~\ref{L:kg facts}, we have
$k_{X_n}=k_{X_n'}$.  The non reductive case then follows directly from
the reductive case.

\section{Comments on other approaches}
\label{sec-comments}

One may wonder about our use of random walks to quantify the
maximality principle for splitting fields, and it is natural to see
why it is interesting, and what other  approaches to ``random''
elements are possible. 
\par
These are essentially of two kinds: one could try to prove upper
bounds for the density
$$
\frac{|\{g\in \Gamma\,\mid\, \|\iota(g)\|\leq X\text{ and
    $\det(T-\iota(g))$ has ``small'' Galois group}\}|}
{|\{g\in \Gamma\,\mid\, \|\iota(g)\|\leq X\}|},
$$
as $X$ grows, where $\iota$ denotes a fixed faithful representation of 
$\G$ into some $GL(n)$ and $\|g\|$ is (say) the Hilbert-Schmidt norm on
$\GL(n,\Cc)$. Or one could still use the system of generators $S$ but
try to bound
$$
\frac{|\{g\in \Gamma\,\mid\, \ell_S(g)\leq X\text{ and
    $\det(T-\iota(g))$ has ``small'' Galois group}\}|}
{|\{g\in \Gamma\,\mid\, \ell_S(g)\leq X\}|},
$$
where $\ell_S(g)$ is the combinatorial distance on $\Gamma$ defined by
$S$. The sieve techniques can potentially extend to these situations,
but one needs to know good equidistribution properties for reduction
modulo primes in these two types of balls, uniformly and
quantitatively. The uniformity will ultimately depend on the spectral
gap property of $\Gamma$ (i.e., on Property $(\tau)$), but due to the
relations in the group, it is not so easy to derive from it the
required equidistribution, in the combinatorial case (one would need
to do it in each coset of $\Gamma^{sc}$, of course). In the
archimedean case, this has very recently been implemented along these
lines by Gorodnik and Nevo~\cite{gn1}, using their deep
ergodic-theoretic equidistribution results~\cite{gn2}.
\par
Moreover, in comparison with these two other approaches, random walks
have one interesting feature: they lend themselves readily to concrete
computations, and in this respect can be pretty efficient. This is
illustrated, in the earlier paper~\cite{jkze8}, by the fairly small
size of the polynomial $P$ with Galois group $W(\Ee_8)$ that we
obtained, especially if the corresponding element of $\Ee_8(\Qq)$ is
expressed as a product of standard Steinberg generators $x_1$, \ldots,
$x_8$: we have simply
$$
P=\det(T-\Ad(x_1\cdots x_8x_1^{-1}\cdots x_8^{-1}))/(T-1)^8\in\Zz[T].
$$
\par
In other words, the complexity of the polynomial (if not of the
splitting field, in terms of usual algebraic invariants such as the
discriminant of the ring of integers, which is difficult to control)
is fairly directly related to the length of the walk.
\par
Another point is that random walks enable us to state some
corollaries, and ask some questions, which do not make sense for other
meanings of ``random'' elements. For instance, given the random walk
$(X_n)$ as in Theorem~\ref{th-main2}, it follows (in the semisimple
case) from the Borel-Cantelli Lemma that, almost surely, there are
only finitely many $n$ for which $\Gal(k_{X_n}/k)$ is not isomorphic
to $W(\G)$. We can then ask how the random variables
\begin{align*}
\tau&=\min\{n\geq 1\,\mid\, \Gal(k_{X_n}/k)=W(\G)\},\\
\tau^*&=\max\{n\geq 1\,\mid\, \Gal(k_{X_n}/k)\not= W(\G)\}
\end{align*}
are distributed?

\end{document}